\RequirePackage{amsthm}
\documentclass[sn-mathphys-num]{sn-jnl}

\usepackage{graphicx}%
\usepackage{multirow}%
\usepackage{amsmath,amssymb,amsfonts}%
\usepackage{amsthm}%
\usepackage{mathrsfs}%
\usepackage[title]{appendix}%
\usepackage{xcolor}%
\usepackage{textcomp}%
\usepackage{manyfoot}%
\usepackage{booktabs}%
\usepackage{algorithm}%
\usepackage{algorithmicx}%
\usepackage{algpseudocode}%
\usepackage{listings}%

\usepackage{adjustbox}
\usepackage{bm}
\usepackage{color}
\usepackage{comment}
\usepackage{pgf}
\usepackage{pgfplots}
\usepgfplotslibrary{groupplots}
\usepackage[labelformat=simple]{subcaption}

\usepackage{tikz}
\usepackage[normalem]{ulem}
\usetikzlibrary{matrix,patterns,positioning,decorations,arrows,shapes,decorations.markings,calc}

\usepackage[symbol]{footmisc}

\theoremstyle{thmstyleone}%
\newtheorem{theorem}{Theorem}
\newtheorem{assumption}[theorem]{Assumption}
\newtheorem{proposition}[theorem]{Proposition}%
\newtheorem{corollary}[theorem]{Corollary}

\theoremstyle{thmstyletwo}%
\newtheorem{example}{Example}%

\theoremstyle{thmstylethree}%
\newtheorem{definition}{Definition}%

\def \bR {\mathbb{R}}

\begin{document}

\title[Article Title]{Efficient Branching Rules for Optimizing Range and Order-Based Objective Functions\footnote[2]{This is a revised and extended version of a conference paper presented at IPCO 2024 \cite{vanrossum2024range}.}}


\author*[1]{\fnm{Bart} \spfx{van} \sur{Rossum}}\email{vanrossum@ese.eur.nl}

\author*[2]{\fnm{Rui} \sur{Chen}}\email{rchen@cuhk.edu.cn}

\author[3]{\fnm{Andrea} \sur{Lodi}}\email{andrea.lodi@cornell.edu}

\affil*[1]{\orgdiv{Econometric Institute}, \orgname{Erasmus University Rotterdam}, \orgaddress{\city{Rotterdam}, \country{the Netherlands}}}

\affil[2]{\orgname{The Chinese University of Hong Kong, Shenzhen}, \orgaddress{Guangdong, China}}

\affil[3]{\orgname{Cornell Tech}, \orgaddress{\city{New York}, \country{USA}}}


\abstract{We consider range minimization problems featuring exponentially many variables, as frequently arising in fairness-oriented or bi-objective optimization. While branch and price is successful at solving cost-oriented problems with many variables, the performance of classical branch-and-price algorithms for range minimization is drastically impaired by weak linear programming relaxations. We propose range branching, a generic branching rule that directly tackles this issue and can be used on top of problem-specific branching schemes. We show several desirable properties of range branching and show its effectiveness on a series of instances of the fair capacitated vehicle routing problem and fair generalized assignment problem. Range branching significantly improves multiple classical branching schemes in terms of computing time, optimality gap, and size of the branch-and-bound tree, allowing us to solve many more large instances than classical methods. Moreover, we show how range branching can be successfully generalized to order-based objective functions, such as the Gini deviation.}

\keywords{Range minimization, Branch and price, Vehicle routing, Generalized assignment, Fairness}



\maketitle

\section{Introduction}
\label{sec:intro}

A growing stream of literature focuses on fairness-oriented optimization \cite{jozefowiez2009evolutionary,matl2018workload,matl2019workload,rossum2023optimizing,tsang2023unified,xinying2023guide, karsu2015inequity, lodi2022fairness}, where not only the efficiency but also the fairness of solutions are taken into account. A popular fairness measure is the range, defined as the maximum minus the minimum of the quantities of interest \cite{matl2018workload,matl2019workload,rossum2023optimizing,tsang2023unified}. In many practical applications, the range and hence the fairness of the solution can be significantly improved with a minor loss in efficiency. Relevant examples include vehicle routing \cite{matl2019workload, rossum2023optimizing}, crew rostering \cite{breugem2022equality}, order picking \cite{vanheusden2020operational}, and the $p$-median problem \cite{ljubic2024benders,marin2020fresh}. Indeed, Figure~\ref{fig:cvrp} presents an example where the range between the distances of vehicle routes can be reduced by 26.2\% by tolerating an efficiency loss of 0.05\% compared to the most cost-efficient solution. In addition to fairness applications, range minimization is often encountered as one of the objectives in bi-objective optimization \cite{jozefowiez2009evolutionary} or even as a stand-alone problem \cite{puerto2012range}. 

\begin{figure}[htb!]
\centering
\begin{subfigure}[t]{0.48\textwidth}
\centering
\begin{tikzpicture}[scale=1]
\begin{axis}[
    axis lines = none,
    width = \textwidth, 
    height = \textwidth,
    disabledatascaling,
    xmin=-50, xmax=950, ymin=-50, ymax=950,
    ]

 \tikzset{every node/.style={draw=black, fill=black!15, shape=circle, scale=0.5}}

        \node[shape=rectangle, scale=2, fill=white] at (854, 504) (0) { };
        \node at (236, 176) (1) { };
        \node at (377, 404) (2) { };
        \node at (380, 468) (3) { };
        \node at (570, 752) (4) { };
        \node at (620, 146) (5) { };
        \node at (616, 746) (6) { };
        \node at (675, 114) (7) { };
        \node at (556, 640) (8) { };
        \node at (99, 777) (9) { };
        \node at (12, 194) (10) { };
        \node at (920, 116) (11) { };
        \node at (846, 384) (12) { };
        \node at (525, 23) (13) { };
        \node at (199, 348) (14) { };
        \node at (813, 875) (15) { };
 
	\draw[draw = black, densely dashed] (0) -- (1);
        \draw[draw = black, densely dashed] (1) -- (10);
        \draw[draw = black, densely dashed] (10) -- (14);
        \draw[draw = black, densely dashed] (14) -- (9);
        \draw[draw = black, densely dashed] (9) -- (0);
        \draw[draw = black, densely dashed] (0) -- (8);
        \draw[draw = black, densely dashed] (8) -- (3);
        \draw[draw = black, densely dashed] (3) -- (2);
        \draw[draw = black, densely dashed] (2) -- (0);
        \draw[draw = black, densely dashed] (0) -- (11);
        \draw[draw = black, densely dashed] (11) -- (7);
        \draw[draw = black, densely dashed] (7) -- (13);
        \draw[draw = black, densely dashed] (13) -- (5);
        \draw[draw = black, densely dashed] (5) -- (0);
        \draw[draw = black, densely dashed] (0) -- (12);
        \draw[draw = black, densely dashed] (12) -- (0);
        \draw[draw = black, densely dashed] (0) -- (4);
        \draw[draw = black, densely dashed] (4) -- (6);
        \draw[draw = black, densely dashed] (6) -- (15);
        \draw[draw = black, densely dashed] (15) -- (0);
\end{axis}
\end{tikzpicture}
\caption{Total distance: 6,178. Range: 2,170.}
\end{subfigure}
\hspace{0.25cm}
\begin{subfigure}[t]{0.48\textwidth}
\centering
\begin{tikzpicture}[scale = 1]
\begin{axis}[
    axis lines = none,
    width = \textwidth, 
    height = \textwidth,
    disabledatascaling,
    xmin=-50, xmax=950, ymin=-50, ymax=950,
    ]

 \tikzset{every node/.style={draw=black, fill=black!15, shape=circle, scale=0.5}}

        \node[shape=rectangle, scale=2, fill=white] at (854, 504) (0) { };
        \node at (236, 176) (1) { };
        \node at (377, 404) (2) { };
        \node at (380, 468) (3) { };
        \node at (570, 752) (4) { };
        \node at (620, 146) (5) { };
        \node at (616, 746) (6) { };
        \node at (675, 114) (7) { };
        \node at (556, 640) (8) { };
        \node at (99, 777) (9) { };
        \node at (12, 194) (10) { };
        \node at (920, 116) (11) { };
        \node at (846, 384) (12) { };
        \node at (525, 23) (13) { };
        \node at (199, 348) (14) { };
        \node at (813, 875) (15) { };
 
	\draw[draw = black, dashed] (0) -- (12);
        \draw[draw = black, dashed] (12) -- (0);
        \draw[draw = black, dashed] (0) -- (1);
        \draw[draw = black, dashed] (1) -- (10);
        \draw[draw = black, dashed] (10) -- (14);
        \draw[draw = black, dashed] (14) -- (2);
        \draw[draw = black, dashed] (2) -- (0);
        \draw[draw = black, dashed] (0) -- (3);
        \draw[draw = black, dashed] (3) -- (9);
        \draw[draw = black, dashed] (9) -- (8);
        \draw[draw = black, dashed] (8) -- (0);
        \draw[draw = black, dashed] (0) -- (4);
        \draw[draw = black, dashed] (4) -- (6);
        \draw[draw = black, dashed] (6) -- (15);
        \draw[draw = black, dashed] (15) -- (0);
        \draw[draw = black, dashed] (0) -- (5);
        \draw[draw = black, dashed] (5) -- (13);
        \draw[draw = black, dashed] (13) -- (7);
        \draw[draw = black, dashed] (7) -- (11);
        \draw[draw = black, dashed] (11) -- (0);
\end{axis}
\end{tikzpicture}
\caption{Total distance: 6,181. Range: 1,601.}
\end{subfigure}
\caption{Efficient (left) and fair (right) vehicle routing solutions.}
\label{fig:cvrp}
\end{figure}
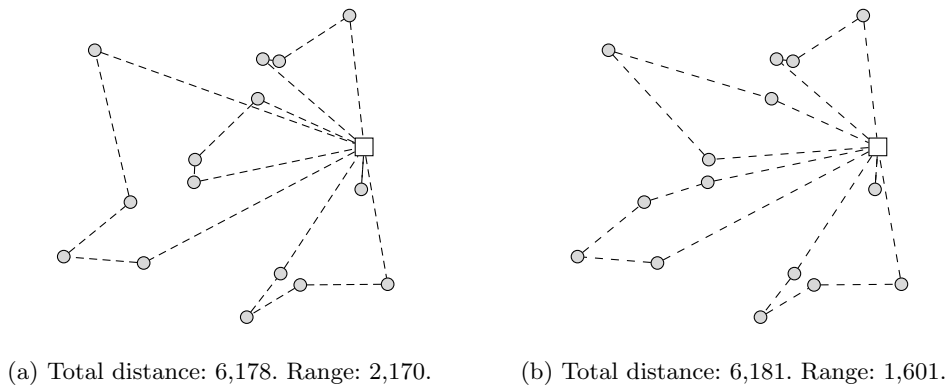

We focus on range minimization problems featuring a number of columns that is prohibitively large for mixed-integer linear programming solvers (MILP-solvers, e.g., Gurobi \cite{gurobi}). This is commonly observed in railway crew scheduling, vehicle routing, and airline crew pairing problems \cite{desaulniers2006column}. Here, formulations with a huge number of variables are preferred as they typically display a tighter linear programming (LP) relaxation and a less symmetric structure than their compact counterparts \cite{barnhart1998branch}. When a classical cost-based objective is used, branch and price is the state-of-the-art solution approach for this type of problems. At each node of the branch-and-bound tree, the LP relaxation is solved through column generation. Here, only a small subset of all possible columns is initially considered, and a pricing algorithm is used to solve the pricing problem of identifying columns that potentially improve the objective value. A branching scheme is used to obtain integer solutions, and all branching decisions must be compatible with the pricing algorithm. We refer to \cite{barnhart1998branch} for a detailed introduction to branch and price.

Unfortunately, formulations for range minimization problems often suffer from poor LP relaxations that drastically impair the performance of classical branch-and-price algorithms. Consider, for example, the capacitated vehicle routing problem (CVRP) with route balancing, which can be phrased as a bi-objective problem where range minimization is one of the objectives. While the branch-and-price (-and-cut) scheme is highly effective for the classical cost-oriented CVRP \cite{costa2019exact}, the range minimization problem is known to suffer from very weak dual bounds \cite{rossum2023optimizing}.

The phenomenon of weak dual bounds has also been repeatedly observed in min-max routing problems, where the goal is to minimize the length of the longest route. In a sense, these problems are special cases of range minimization problems, with a zero objective coefficient on the length of the shortest route. Examples include the min-max multiple traveling salesman problem \cite{francca1995minmax} and the min-max open vehicle routing problem \cite{applegate2002solution,lysgaard2020matheuristic,marinho2021}. For min-max routing problems, the bisection method (a form of dichotomous search) can be efficiently applied to obtain optimal solutions on instances with up to 150 nodes \cite{francca1995minmax,marinho2021}. Unfortunately, this approach cannot be generalized to range minimization: since the solution to range minimization problems is characterized by both a minimum and maximum payoff, binary search along one dimension is not possible. As a result, enumeration schemes \cite{matl2019workload} and a wide range of heuristics have been proposed for the fairness-oriented CVRP (see \cite{matl2018workload} for an overview).

Our main contribution is to propose an efficient exact branch-and-price framework for large-scale range minimization problems. In particular, we propose a specialized branching rule, which we refer to as \emph{range branching}, that directly tackles the poor LP relaxation of range minimization formulations. It can be used on top of problem-specific branching schemes, and allows one to effectively leverage branch-and-price techniques commonly applied to cost-oriented optimization problems. We show some promising properties of range branching, and evaluate its effectiveness on a series of fairness-oriented capacitated vehicle routing and generalized assignment problems. Range branching significantly improves the performance of multiple classical branching schemes in terms of computing time, optimality gap, and size of the branch-and-bound tree, allowing us to solve many more instances to optimality and outperform a standard MILP-solver. This result appears to be mainly driven by strongly improved dual bound behaviour. Finally, we successfully show how range branching can be generalized to order-based minimization problems, and use this result to minimise the Gini deviation in CVRP.

The remainder of this paper is structured as follows. We formally define the range minimization problem in Section~\ref{sec:problem} and present range branching in Section~\ref{sec:branching}. In Sections~\ref{sec:cvrp} and~\ref{sec:gap}, we present the results of computational experiments on the fair capacitated vehicle routing problem and fair generalized assignment problem, respectively. We extend our range branching rule to the case of general order-based objective functions in Section~\ref{sec:order}, and conclude in Section~\ref{sec:conclusion}.

\section{Range Minimization Problem}
\label{sec:problem}

We formally define the range minimization problem in Section~\ref{subsec:problem_description}, and discuss mathematical formulations for this problem in Section~\ref{subsec:problem_formulation}.

\subsection{Problem Description}
\label{subsec:problem_description}
Throughout this work, we consider problems with the following structure. Given a set of $N$ columns (implicitly), the problem requires to select a subset of them, with decision variables $\bm{x} = (x_1, \dots, x_N)\in \{0, 1\}^N$ being the indicator variables of the selected columns.\footnote[2]{For simplicity, we do not assume additional auxiliary variables. However, these can be easily incorporated into our approach.} We assume that some linear constraints $A \bm{x} \leq \bm{b}$ must be respected. These constraints can include efficiency lower bounds, constraints on the number of selected columns, and other problem-specific constraints (e.g., in the case of the CVRP, regular customer partition constraints). To summarize, the feasible region of the decision variables is given by $\mathcal{X} = \{ \bm{x} \in \{0, 1\}^N : A \bm{x} \leq \bm{b} \}$, and the feasible region of its linear programming relaxation is given by $\mathcal{X}_{LP} = \{ \bm{x} \in [0, 1]^N : A \bm{x} \leq \bm{b} \}$.

The objective is to minimize the range of the payoffs of selected columns. In particular, we assume that each column $i$ has an associated payoff value $p_i$. Without loss of generality, we assume nonnegative and bounded payoffs, i.e., for all $i$ we have $p_i \in [0, M]$ for some (large) constant $M$. For notational convenience, given $\bm{x}\in[0,1]^N$ and $\bm{p} \in [0, M]^N$, define $p_{\min}(\bm{x}) = \min_{i : x_i > 0} p_i$ and $p_{\max}(\bm{x}) = \max_{i : x_i > 0} p_i$ to be the minimum and maximum payoff of any (possibly fractional) solution $\bm{x}$, respectively. We use a definition that covers fractional solutions because these are frequently encountered during the course of branch and price. The goal of the range minimization problem is to determine a binary solution $\bm{x} \in \mathcal{X}$ minimizing the difference between the largest and smallest payoffs, i.e., minimizing $p_{\max}(\bm{x}) - p_{\min}(\bm{x})$. 

Finally, we assume that the large number of columns $N$ prohibits the out-of-the-box use of standard mixed-integer programming solvers, and that an exact branch-and-price algorithm is used to solve the problem. In particular, we assume that a problem-specific pricing algorithm and  a problem-specific branching scheme, i.e., set of branching rules, are available. The range branching rule we propose will operate \textit{on top of} the problem-specific branching scheme. In other words, the range branching rule is to be invoked first at any node in the branch-and-price tree, and the regular branching scheme is applied only if range branching does not detect any branching opportunities. Since the range branching decisions induce constraints involving the payoff value $p_i$, as well as constraints limiting the domain of $p_i$, we assume that these constraints are amenable to the pricing algorithm. We will show that this is the case for the vehicle routing and generalized assignment applications under consideration. 

\subsection{Mathematical Formulation}
\label{subsec:problem_formulation}

We now discuss a general class of formulations for the range minimization problem. To this end, we first introduce the concept of \emph{range-respecting solutions}.
\begin{definition}[Range-Respecting Solution] 
Let a vector $(\bar{\bm{x}}, \bar{\eta}, \bar{\gamma})\in[0,1]^N\times\bR^{2}$ be given. We say this vector is \emph{range-respecting} (with respect to $\bm{p}$) if $\bar{\eta} \geq p_{\max}(\bar{\bm{x}})$ and $\bar{\gamma} \leq p_{\min}(\bar{\bm{x}})$, and \emph{range-violating} (with respect to $\bm{p}$) otherwise.
\label{def:range_respecting}
\end{definition}
The range of a solution $\bar{\bm{x}}\in\mathcal{X}$ is equal to $\bar{\eta}-\bar{\gamma}$ with the smallest $\bar{\eta}$ and largest $\bar{\gamma}$ that make $(\bar{\bm{x}}, \bar{\eta}, \bar{\gamma})$ range-respecting. We are now ready to define a class of formulations for the range minimization problem.
\begin{definition}[Valid Formulation]
Let $A, \bm{b}$ be given. Let $\bm{w} = (C, \bm{d}, \bm{f}, \bm{g})$ be a tuple of constraint matrices and coefficient vectors, and consider the following mixed-binary linear program:
\begin{subequations}\label{range_min}
\begin{align}
\min_{\bm{x}, \eta, \gamma} \quad & \eta - \gamma \label{eq:obj} \\
\text{s.t.}\quad & A \bm{x} \leq \bm{b} \\ 
& C \bm{x} + \eta \bm{d}+\gamma \bm{f} \leq \bm{g} \label{eq:auxiliary_min_max} \\
& \bm{x} \in \{0, 1\}^N. \label{eq:domain}  
\end{align}
\end{subequations}
\label{def:range_problem}
Auxiliary constraints~(\ref{eq:auxiliary_min_max}), as defined by  $\bm{w}$, model the values of $\eta$ and $\gamma$ representing the maximum and minimum payoff. We say program~(\ref{range_min}) is a \emph{valid formulation} for the range minimization problem when $\bm{z} = (\bm{x}, \eta, \gamma)$ is feasible in~(\ref{range_min}) if and only if $\bm{x} \in \mathcal{X}$ and $\bm{z}$ is range-respecting. We say $\bm{w}$ is valid when it defines a valid formulation. We denote by $\mathcal{Z}_{LP}(\bm{w})$ the feasible region of the LP relaxation of~(\ref{range_min}), obtained by relaxing constraints~(\ref{eq:domain}) to $\bm{x} \in [0, 1]^N$.
\end{definition}
Without loss of generality, we can assume that, for any solution in the LP relaxation $\mathcal{X}_{LP}$ of feasible region $\mathcal{X}$, its range-respecting counterpart is contained in the LP relaxation of any valid formulation.
\begin{assumption}[Fractional Range-Respecting Solutions]
Let $\bm{w}$ be valid. Then, for any $\bm{x} \in \mathcal{X}_{LP}$, the range-respecting solution $\bm{z} = (\bm{x}, p_{\text{max}}(\bm{x}), p_{\text{min}}(\bm{x}))$ is contained in $\mathcal{Z}_{LP}(\bm{w})$.
\label{assumption:fractional}
\end{assumption}
Suppose, to the contrary, that the auxiliary constraints modeling the range tighten the description of the feasible region and cut off some fractional $\bm{x}$. Then, these auxiliary constraints could have been included as valid inequalities in the set $A \bm{x} \leq \bm{b}$.

Ideally, we wish to minimize over the set of range-respecting solutions only. Unfortunately, this goal is well out of reach. Since the operators $p_{\text{max}}$ and $p_{\text{min}}$ are non-convex and non-concave, any valid formulation of a nontrivial range minimization problem contains infinitely many range-violating fractional solutions, as demonstrated in our next result.
\begin{proposition}[Fractional Range-Violating Solutions]
Let $\mathcal{X}$ be such that there exist distinct $\bm{x}_1, \bm{x}_2 \in \mathcal{X}$ for which $p_{\text{max}}(\bm{x}_1) > p_{\text{max}}(\bm{x}_2)$ or $p_{\text{min}}(\bm{x}_1) > p_{\text{min}}(\bm{x}_2)$. Then, for any valid $\bm{w}$, it holds that $\mathcal{Z}_{LP}(\bm{w})$ contains infinitely many range-violating fractional solutions.
\end{proposition}

\begin{proof}
Let a valid $\bm{w}$ be given, and suppose that $p_{\text{max}}(\bm{x}_1) > p_{\text{max}}(\bm{x}_2)$ for some $\bm{x}_1, \bm{x}_2 \in \mathcal{X}$. Since $\bm{w}$ is valid, we know that the range-respecting integral solutions $\bm{z}_1 = (\bm{x}_1, p_{\text{max}}(\bm{x}_1), p_{\text{min}}(\bm{x}_1))$ and $\bm{z}_2 = (\bm{x}_2, p_{\text{max}}(\bm{x}_2), p_{\text{min}}(\bm{x}_2))$ are contained in $\mathcal{Z}_{LP}(\bm{w})$. Now, for $\lambda \in (0, 1)$, define $\bm{z}(\lambda) = (\bm{x}(\lambda), \eta(\lambda), \gamma(\lambda)) = \lambda \bm{z}_1 + (1 - \lambda) \bm{z}_2$. By convexity of $\mathcal{Z}_{LP}(\bm{w})$, we know that $\bm{z}(\lambda) \in \mathcal{Z}_{LP}(\bm{w})$. By construction, $p_{\text{max}}(\bm{x}_\lambda) = p_{\text{max}}(\bm{x}_1) > \lambda p_{\text{max}}(\bm{x}_1)+(1-\lambda)p_{\text{max}}(\bm{x}_2)=\eta(\lambda)$ for $\lambda \in (0, 1)$. In other words, $\bm{z}(\lambda)\in \mathcal{Z}_{LP}(\bm{w})$ is range-violating for $\lambda \in (0, 1)$. The case $p_{\text{min}}(\bm{x}_1) > p_{\text{min}}(\bm{x}_2)$ is treated similarly.
\end{proof}
The admission of many range-violating fractional solutions can result in poor dual bounds, thereby hindering the performance of branch-and-price algorithms. As we will see, range branching provides a way of eliminating range-violating solutions from valid formulations.

While constructing a tractable valid formulation is highly problem-dependent, we like to make some general remarks on how to formulate constraints that correctly link $\bm{x}$ to $\eta$ and $\gamma$. First, constraints per individual column of the type $p_i x_i \leq \eta$ are valid for \textit{any} range minimization problem, but are highly impractical. The use of such constraints would require simultaneous column-and-row generation, severely complicating the pricing problem. Instead, one typically resorts to some form of aggregated constraints. A notable case arises when the set of columns can be written as some union of subsets, and exactly one column is selected from each subset in any feasible solution. This situation is frequently encountered in agent-based problems, where each agent is always assigned exactly one `job' (column). Formally, let $I$ denote the set of column indices and suppose that $I = \bigcup_{j \in [m]} I_j$ for some $m \in \mathbb{N}_{+}$, where the $I_j$ need not be pairwise disjoint and $[m] = \{1, \dots, m\}$ denotes the set of integers ranging from $1$ to $m$. If, for any feasible $\bm{x} \in \mathcal{X}$ and all $j \in [m]$, it holds that $\sum_{i \in I_j} x_i = 1$, the following constraints correctly model the maximum and the minimum:
\begin{subequations}
\begin{align}
& \sum_{i \in I_j} p_i x_i \leq \eta && \forall j \in [m] \\ 
& \sum_{i \in I_j} p_i x_i \geq \gamma && \forall j \in [m]. 
\end{align}
\label{eq:general_min_max}
\end{subequations}
In some applications, we might only be able to show that \textit{at most} one column is selected from each subset. If, for any feasible $\bm{x} \in \mathcal{X}$ and all $j \in [m]$, it holds that $\sum_{i \in I_j} x_i \leq 1$, we must replace constraints~(\ref{eq:general_min_max}) by the following, slightly weaker set of constraints:
\begin{subequations}
\begin{align}
\sum_{i \in I_j} p_i x_i \leq \eta && \forall j \in [m] \\ 
M + \sum_{i \in I_j} (p_i - M) x_i \geq \gamma && \forall j \in [m].
\end{align}
\label{eq:general_min_max_weak}
\end{subequations}
Recall that $M$ is an upper bound on the payoff of any column. All valid formulations presented in this work make use of some version of either constraints~(\ref{eq:general_min_max}) or~(\ref{eq:general_min_max_weak}). Of course, identifying sets $I_j$ requires problem-specific knowledge and might not always be possible. Finally, problem-specific valid inequalities are unlikely to be of use, as they are typically aimed at tightening the formulation of the problem-specific feasible region $\mathcal{X}$ and do not necessarily cut off fractional range-violating solutions. 

We finish the discussion on valid formulations with a small example, illustrating how a nontrivial gap between the optimal solution and the LP relaxation bound can exist even when $\mathcal{X}_{LP}$ equals the convex hull of $\mathcal{X}$. 
\begin{example}
Consider three columns $x_1, x_2$, and $x_3$ with payoffs $p_1 = 1$,  $p_2 =2$, and $p_3 = 3$, respectively. In any feasible solution, two out of three columns are selected, i.e., $\mathcal{X} = \{ (1, 1, 0), (1, 0, 1), (0, 1, 1)\}$. Clearly, this problem has two optimal solutions, $(1, 1, 0)$ and $(0, 1, 1)$, both of which have a range of one.

We now analyze the dual bound of a valid formulation for this problem. Inspired by constraints~(\ref{eq:general_min_max_weak}), we consider the following mixed-binary linear program:
\begin{subequations}
\begin{align}
\min_{\bm{x}, \eta, \gamma} \quad & \eta - \gamma \\
\text{s.t.}\quad & x_1 + x_2 + x_3 = 2 \\ 
& p_i x_i \leq \eta && \forall i \in [3] \\ 
& 3 + (p_i - 3) x_i \geq \gamma && \forall i \in [3] \\
& \eta \geq \gamma \\ 
& x_1, x_2, x_3 \in \{0, 1\} \\ 
& \eta, \gamma \in [0, 3].
\end{align}
\label{eq:example}
\end{subequations}
Note that we have bounded the domain of $\eta$ and $\gamma$ and added the valid inequality $\eta \geq \gamma$. 

Figure~\ref{fig:example} illustrates the feasible region of model~(\ref{eq:example}). As seen in Figure~\ref{subfig:example_x}, the formulation is a perfect formulation for $\mathcal{X}$, i.e., the set of $\bm{x} \in \mathcal{X}_{LP}$ are exactly those solution vectors in the convex hull of solutions in $\mathcal{X}$. Nonetheless, the set of feasible pairs of $(\eta, \gamma)$, as shown in Figure~\ref{subfig:example_range}, contain a solution with $\eta = \gamma$. In particular, solution $(x_1, x_2, x_3, \eta, \gamma) = (1/2, 1, 1/2, 2, 2)$ is feasible and has range zero. 
\def\nudge{.5}
\tikzset{line/.style={thick, black}}

\begin{figure}[tb!]
\centering
\begin{subfigure}[t]{0.48\textwidth}
    \centering
    \begin{tikzpicture}
    
    \draw[->, thick] (0,0)--(3.5,0) node[right]{$x_1$};
    \draw[->, thick] (0,0)--(0, 3.5) node[above]{$x_2$};

    \draw[thick] (0, -0.1) node[below]{$0$}--(0, 0.1);
    \draw[thick] (3, -0.1) node[below]{$1$}--(3, 0.1);

    \draw[thick] (-0.1, 0) node[left]{$0$}--(0.1, 0);
    \draw[thick] (-0.1, 3) node[left]{$1$}--(0.1, 3);

    \begin{scope}
    \clip (-\nudge ,-\nudge) rectangle (4+\nudge,4+\nudge);
    
    \draw[thin] (3, 0) -- (3, 3) -- (0,3) -- (3, 0);
    \clip (3, 0) -- (3, 3) -- (0,3) -- (3, 0);
    \fill[gray!15] (0,0) rectangle (4,4);
    \end{scope}

    \filldraw[black] (3, 3) circle (2pt); 
    \filldraw[black] (0, 3) circle (2pt); 
    \filldraw[black] (3, 0) circle (2pt); 
    
    \end{tikzpicture}
    \caption{Projection onto space of $(x_1, x_2)$.}
    \label{subfig:example_x}
\end{subfigure}
\hspace{0.25cm}
\begin{subfigure}[t]{0.48\textwidth}
    \centering
    \begin{tikzpicture}
    
    \draw[->, thick] (0,0)--(3.5,0) node[right]{$\eta$};
    \draw[->, thick] (0,0)--(0, 3.5) node[above]{$\gamma$};

    \draw[thick] (0, -0.1) node[below]{$0$}--(0, 0.1);
    \draw[thick] (1, -0.1) node[below]{$1$}--(1, 0.1);
    \draw[thick] (2, -0.1) node[below]{$2$}--(2, 0.1);
    \draw[thick] (3, -0.1) node[below]{$3$}--(3, 0.1);

    \draw[thick] (-0.1, 0) node[left]{$0$}--(0.1, 0);
    \draw[thick] (-0.1, 1) node[left]{$1$}--(0.1, 1);
    \draw[thick] (-0.1, 2) node[left]{$2$}--(0.1, 2);
    \draw[thick] (-0.1, 3) node[left]{$3$}--(0.1, 3);
    
    \begin{scope}
    \clip (-\nudge ,-\nudge) rectangle (4+\nudge,4+\nudge);
    
    \draw[thin, line] (1.2, 0) -- (1.2, 1) -- (1.5, 1.5) -- (2.14, 2.14) -- (3, 2.33) -- (3, 0);

    \draw[dashed, line] (0, 0) -- (3.5, 3.5) node[left]{$\eta = \gamma$};
    \clip (1.2, 0) -- (3, 0) -- (3, 2.33) -- (2.14, 2.14) -- (1.5, 1.5) -- (1.2, 1) -- (1.2, 0);
    \fill[gray!15] (0,0) rectangle (4,4);
    \end{scope}

    \filldraw[black] (3, 1) circle (2pt); 
    \filldraw[black] (3, 2) circle (2pt); 
    \filldraw[black] (2, 1) circle (2pt); 
    \filldraw[black] (3, 0) circle (2pt); 
    \filldraw[black] (2, 0) circle (2pt); 
    
    \end{tikzpicture}
    \caption{Projection onto space of ($\eta, \gamma)$.}
    \label{subfig:example_range}
\end{subfigure}

\caption{Feasible region in the space of $(x_1, x_2)$ (left) and $(\eta, \gamma)$ (right). Integer-feasible points are indicated by black dots, and  the feasible region of the LP relaxation is indicated in grey.}
\label{fig:example}
\end{figure}
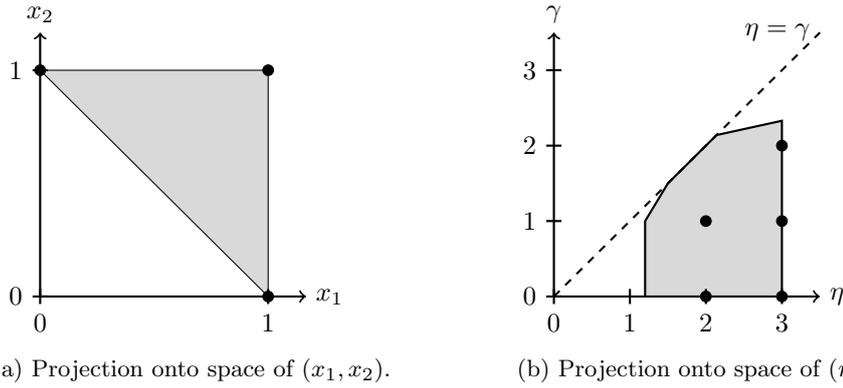

\end{example}

\section{Range Branching}
\label{sec:branching}

We propose \emph{range branching}, a generic branching strategy for range minimization problems. We define the range branching rule, to be used on top of problem-specific branching schemes, in Section~\ref{subsec:branching_rule}, and show some desirable properties of range branching in Section~\ref{subsec:branching_theory}. In Section~\ref{subsec:branching_practical}, we discuss some practical considerations regarding the implementation of range branching.

\subsection{Range Branching Rule}
\label{subsec:branching_rule}

The goal of range branching is to eliminate all range-violating fractional solutions from consideration and thereby improve the dual bound in branch-and-price algorithms for range minimization problems. To this end, we propose to (i) branch on the values of $\eta$ and $\gamma$ and (ii) make use of a combination of variable bounding and variable fixing.

We now formally define the range branching rule. Consider a fractional solution $(\bar{\bm{x}}, \bar{\eta}, \bar{\gamma})$ encountered at any node in the branch-and-price tree. Recall that $p_{\min}(\bar{\bm{x}}) = \min_{i : \bar{x}_i > 0} p_i$ and $p_{\max}(\bar{\bm{x}}) = \max_{i : \bar{x}_i > 0} p_i$. When this solution is range-respecting, i.e., $\bar\gamma\leq p_{\min}(\bar{\bm{x}})\leq p_{\max}(\bar{\bm{x}})\leq\bar\eta$, our branching rule does not create any branches. Instead, the problem-specific branching scheme is invoked to construct branches that cut-off the incumbent fractional solution. Otherwise, suppose that the incumbent fractional solution $(\bar{\bm{x}}, \bar{\eta}, \bar{\gamma})$ is range-violating, for example, $\bar{\eta} < p_{\max}(\bar{\bm{x}})$. We cut off this solution by branching on $\eta$. In particular, we select a cutoff value $U \in (\bar{\eta}, p_{\max}(\bar{\bm{x}}))$ and create two child nodes. In the left child node, we impose $\eta \leq U$. In addition, in the left child node and its descendants, we apply variable fixing by setting to zero all $x_i$ for which $p_i > U$ and prevent the generation of columns with payoffs $p_i > U$ in the pricing problem. In other words, we apply the locally valid inequality $\sum_{i : p_i > U} x_i = 0$ in the left node. In the right child node, we only impose $\eta \geq U$. We have experimented with enforcing the locally valid inequality $\sum_{i : p_i \geq U} x_i \geq 1$, but this did not improve the computational performance. The case $\bar{\gamma} > p_{\min}(\bar{\bm{x}})$ is treated similarly. This branching rule preserves all range-respecting solutions, but cuts off range-violating incumbent solutions. Note that range branching is to be applied at every node in the branch-and-price tree, even those where problem-specific branching decisions are already being enforced.

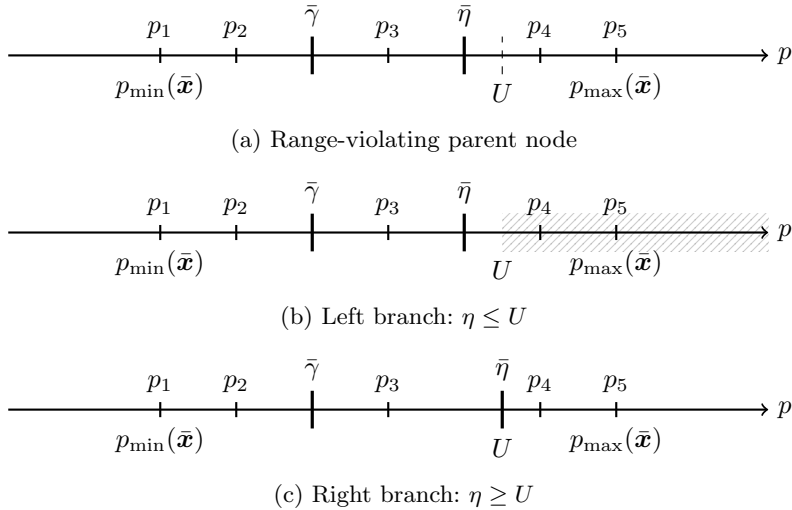
\begin{figure}[tb!]
\centering
\begin{subfigure}{\textwidth}
    \centering
    \begin{tikzpicture}
    \draw[->, thick] (-5,0)--(5,0) node[right]{$p$};

    \draw[very thick] (-1,-0.25)--(-1,0.25) node[above]{$\bar{\gamma}$};
    \draw[very thick] (1,-0.25)--(1,0.25) node[above]{$\bar{\eta}$};

    \draw[thick] (-3,-0.1) node[below]{$
    p_{\min}(\bar{\bm{x}})$} --(-3,0.1) node[above]{$p_1$};
    \draw[thick] (-2,-0.1)--(-2,0.1) node[above]{$p_2$};
    \draw[thick] (0,-0.1)--(0,0.1) node[above]{$p_3$};
    \draw[thick] (2,-0.1)--(2,0.1) node[above]{$p_4$};
    \draw[thick] (3,-0.1) node[below]{$
    p_{\max}(\bar{\bm{x}})$} --(3,0.1) node[above]{$p_5$};

    \draw[dashed] (1.5, -0.25) node[below]{$U$} --(1.5, 0.25);
    \end{tikzpicture}
    \caption{Range-violating parent node}
    \label{subfig:violating}
\end{subfigure}

\begin{subfigure}{\textwidth}
    \centering
    \begin{tikzpicture}
    \draw[->, thick] (-5,0)--(5,0) node[right]{$p$};

    \draw[very thick] (-1,-0.25)--(-1,0.25) node[above]{$\bar{\gamma}$};
    \draw[very thick] (1,-0.25)--(1,0.25) node[above]{$\bar{\eta}$};

    \draw[thick] (-3,-0.1) node[below]{$
    p_{\min}(\bar{\bm{x}})$} --(-3,0.1) node[above]{$p_1$};
    \draw[thick] (-2,-0.1)--(-2,0.1) node[above]{$p_2$};
    \draw[thick] (0,-0.1)--(0,0.1) node[above]{$p_3$};
    \draw[thick] (2,-0.1)--(2,0.1) node[above]{$p_4$};
    \draw[thick] (3,-0.1) node[below]{$
    p_{\max}(\bar{\bm{x}})$} --(3,0.1) node[above]{$p_5$};

    \fill[black, pattern=north east lines, opacity=.45] (1.5, -0.25) node[below, opacity=1]{$U$} rectangle (5, 0.25);
    \end{tikzpicture}
    \caption{Left branch: $\eta \leq U$}
    \label{subfig:left}
\end{subfigure}

\begin{subfigure}{\textwidth}
    \centering
   \begin{tikzpicture}
    \draw[->, thick] (-5,0)--(5,0) node[right]{$p$};

    \draw[very thick] (-1,-0.25)--(-1,0.25) node[above]{$\bar{\gamma}$};
    \draw[very thick] (1.5,-0.25) node[below]{$U$} --(1.5,0.25) node[above]{$\bar{\eta}$};

    \draw[thick] (-3,-0.1) node[below]{$
    p_{\min}(\bar{\bm{x}})$} --(-3,0.1) node[above]{$p_1$};
    \draw[thick] (-2,-0.1)--(-2,0.1) node[above]{$p_2$};
    \draw[thick] (0,-0.1)--(0,0.1) node[above]{$p_3$};
    \draw[thick] (2,-0.1)--(2,0.1) node[above]{$p_4$};
    \draw[thick] (3,-0.1) node[below]{$
    p_{\max}(\bar{\bm{x}})$} --(3,0.1) node[above]{$p_5$};
    \end{tikzpicture}
    \caption{Right branch: $\eta \geq U$}
    \label{subfig:right}
\end{subfigure}
\caption{Example of range branching applied to a range-violating solution. The striped rectangle indicates columns, whose payoffs are above the threshold, that have been fixed to zero.}
\label{fig:range_branching}
\end{figure}

Figure~\ref{fig:range_branching} illustrates the range branching rule with a simple example. Figure~\ref{subfig:violating} presents a feasible solution consisting of five columns with payoffs $p_1$ to $p_5$, respectively. This solution is range-violating since $p_4, p_5 > \bar{\eta}$, although a similar violation can be detected for $\bar{\gamma}$. As such, range branching is applied with cutoff value $U \in (\bar{\eta}, p_4)$. The left and right branches created through range branching are presented in Figures~\ref{subfig:left} and~\ref{subfig:right}, respectively. In the left branch, we impose $\eta \leq U$, and forbid the use of columns with payoffs above $U$, as indicated with the striped rectangle. In particular, columns with payoff $p_4$ and $p_5$ are fixed to zero. In the right branch, we impose $\eta \geq U$, ensuring that the objective value of this child node is strictly larger than that of its parent.

With proper choices of the cutoff values (see Section~\ref{subsec:branching_practical}), the use of this branching rule has several effects on the branch-and-price algorithm. First, we typically observe that the branch in which variable fixing is applied becomes infeasible and can be pruned, especially in early branching iterations. The branch where only variable bounding on $\eta$ and $\gamma$ is used encounters a higher objective value and thereby drives up the overall LP bound. Second, the variable fixing has an ambiguous effect on the efficiency of the pricing problems. On the one hand, incorporating the constraints $p_i \leq U$ or $p_i \geq L$ restricts the feasible region of the pricing problem, which can yield a speed-up. On the other hand, actually enforcing this constraint can complicate the pricing problems, e.g., by prohibiting the use of efficient dominance rules in labeling algorithms. Which effect dominates the other is highly specific to the problem and the payoff structure under consideration. 

\subsection{Properties of Range Branching}
\label{subsec:branching_theory}

We define a \emph{range-branching-free (RBF) node solution} to be any solution vector $\bm{z} = (\bar{\bm{x}}, \bar{\eta}, \bar{\gamma})$ encountered at any node in the branch-and-bound tree for which the range branching rule does not detect any branching opportunities. We refer to the corresponding node as \emph{RBF node}. We say $\bm{z}$ is a RBF node solution with respect to $\bm{w}$ when it is obtained by solving the range formulation defined by $\bm{w}$. By construction, all RBF node solutions are range-respecting. 

\begin{proposition}[RBF Node Solutions are Range-Respecting]
Let $\bm{z} = (\bar{\bm{x}}, \bar{\eta}, \bar{\gamma})$ be a RBF node solution with respect to some valid $\bm{w}$. Then, $\bm{z}$ is range-respecting.
\label{prop:leaf}
\end{proposition}
\begin{proof}
Suppose that $\bm{z}$ is not range-respecting. Then, range branching can be invoked. This contradicts the fact that the node is RBF. Hence, $\bm{z}$ must be range-respecting.
\end{proof}
It follows that any RBF node solution is feasible in the LP relaxation of any valid formulation for the range minimization problem.
\begin{corollary}[Formulation-Independence]
Let $\bm{z} = (\bar{\bm{x}}, \bar{\eta}, \bar{\gamma})$ be a RBF node solution with respect to some valid $\bm{w}$. Then, $\bm{z} \in \mathcal{Z}_{LP}(\bm{w}')$ for all valid $\bm{w}'$ that satisfy Assumption~\ref{assumption:fractional}. 
\label{cor:auxiliary}
\end{corollary}
\begin{proof}
Let $\bm{w}'$ be valid. By Definition~\ref{def:range_problem} and Assumption~\ref{assumption:fractional}, the LP-relaxation of the formulation defined by $\bm{w}'$ admits all range-respecting $\bm{z} = (\bm{x}, \eta, \gamma)$ for which $\bm{x} \in \mathcal{X}_{LP}$. Solution $\bm{z}$ is range-respecting by Proposition~\ref{prop:leaf}, and $\bm{x} \in \mathcal{X}_{LP}$ holds since $\bm{w}$ is valid. Hence, $\bm{z} \in \mathcal{Z}_{LP}(\bm{w}')$.
\end{proof}
As a result, the objective value at any RBF node is at least as high as the best possible root node bound of any valid formulation that admits all fractional range-respecting solutions for range minimization.

\begin{proposition}
Let $\bm{z} = (\bar{\bm{x}}, \bar{\eta}, \bar{\gamma})$ be a RBF node solution with respect to some valid $\bm{w}$, and denote its objective value by $v_{RBF} = \bar{\eta} - \bar{\gamma}$. Moreover, let $v^{*}_{LP}(\bm{u}) = \min_{(\bm{x}, \eta, \gamma) \in \mathcal{Z}_{LP}(\bm{u})} (\eta - \gamma)$ denote the LP bound of the valid formulation defined by $\bm{u}$. Then $v_{RBF} \geq v^{*}_{LP}(\bm{w}')$ for all valid $\bm{w}'$ that satisfy Assumption~\ref{assumption:fractional}.
\end{proposition}

\begin{proof}
The result directly follows from Corollary~\ref{cor:auxiliary}.
\end{proof}

Recall that the LP lower bound in branch and price is defined as the minimum objective over all leaf nodes in the branch-and-bound tree. At some point in the course of the branch-and-price algorithm, all leaf nodes in the branch-and-bound tree will be RBF nodes or descendants of RBF nodes. From this point onward, the LP lower bound will be at least as high as the best possible root node bound of any valid formulation. The rationale behind range branching is that this tighter dual bound will reduce the number of nodes to explore in the branch-and-price tree and therefore result in an overall reduction in computing time.

\subsection{Practical Considerations}
\label{subsec:branching_practical}

\paragraph{Cutoff Values}
As discussed in Section~\ref{subsec:branching_rule}, the success of range branching can be partially explained by its differential effect on the child branches it creates. In one branch, we apply both variable bounding and fixing, after which we can typically prune by infeasibility. In the other branch, we apply variable bounding, thereby driving up the LP bound of this branch. These effects are likely to be small, however, when cutoff values close to $\bar{\gamma}$ and $\bar{\eta}$ are used.

In practice, we find that the following relaxation of range branching works well. Let $\alpha \in [0, 1)$. At a node with incumbent solution $(\bar{\bm{x}}, \bar{\eta}, \bar{\gamma})$, set $U_{\alpha} = (1+\alpha)\bar{\eta}$ and $L_{\alpha} = (1-\alpha)\bar{\gamma}$. We check whether columns with payoff above $U_{\alpha}$ or below $L_{\alpha}$ are used, i.e., when range violations of more than a factor $\alpha$ occur. In this case, we perform range branching with $U_{\alpha}$ and $L_{\alpha}$ as cutoff values. A value of $\alpha=0.025$ seems to perform well in our experiments. The results in Section~\ref{subsec:branching_theory} hold only for the case of $\alpha=0$, since only then all range-violating solutions are removed through range branching. From a computational point of view, however, it is interesting to consider layered branching schemes with decreasing values of $\alpha$.

\paragraph{Further Improvements}
The default range branching scheme can be refined in several ways. First, it is clear that the branching scheme can be tightened when all $p_i$ are integer-valued (e.g., branch by creating nodes $\eta\leq U$ and $\eta\geq U+1$ with integer-valued $U$). Second, one could incorporate a restricted master heuristic throughout the branch-and-price algorithm and use valid upper bounds on the range to locally restrict the domain of $\eta$ and $\gamma$. Third, if one uses big-$M$ constraints to model the minimum, as in~(\ref{eq:general_min_max_weak}), this value of $M$ can be tightened throughout the course of the algorithm based on the variable domain of $\gamma$. Fourth, our theoretical analysis suggests that it is possible to solve range minimization problems through range branching without the use of auxiliary constraints of the type~(\ref{eq:auxiliary_min_max}), i.e., without an explicit valid formulation.

\section{Fair Capacitated Vehicle Routing Problem}
\label{sec:cvrp}

In this section, we introduce an application of the range minimization problem which we call the \emph{fair capacitated vehicle routing problem (F-CVRP)}, a fairness-oriented version of the capacitated vehicle routing problem.

\begin{definition}[Fair Capacitated Vehicle Routing Problem]
Let $C$ be a set of customers, where customer $i \in C$ has integer demand $d_i >0$. Let $G=(\{0\} \cup C, A)$ be the complete directed graph on the set of customers and a central depot, and let $c_a$ and $p_a$ denote the cost and distance of traversing arc $a \in A$, respectively. Moreover, let $K$ homogeneous vehicles with capacity $Q$ be given, and let $B$ be an upper bound on the total cost of all routes. A set of $K$ routes is feasible when each route starts and ends at the depot, each customer is visited by exactly one route, the sum of customer demands along each route does not exceed $Q$, and the total cost of traversed arcs does not exceed $B$. The goal of the fair capacitated vehicle routing problem is to select a feasible set of routes that minimizes the range of the route distances.
\label{def:fcvrp}
\end{definition}
We include an upper bound on the overall route cost to ensure that solutions do not deteriorate too much compared to the cost-efficient solution. Without such a bound, the length of the shortest route can be artificially lengthened to increase the minimum and thereby reduce the range. In practice, a tight efficiency bound leaves little room for such routes, but there is no guarantee that selected routes are TSP-optimal: shorter routes visiting the same set of customers in different order might exist. 

We present two mathematical formulations of F-CVRP in Section~\ref{subsec:cvrp_math} and present a branch-and-price algorithm with range branching in Section~\ref{subsec:cvrp_branch}. We describe the set-up of our experiments in Section~\ref{subsec:cvrp_setup} and discuss the results in Section~\ref{subsec:cvrp_results}. We empirically study the computational cost of enforcing TSP-optimality in the F-CVRP, and analyze whether our method can effectively solve this problem variant in Section~\ref{subsec:cvrp_tsp}.

\subsection{Mathematical Formulation}
\label{subsec:cvrp_math}

We present two formulations for the F-CVRP. The vehicle-index formulation explicitly assigns routes to vehicles. The last-customer formulation does not require a vehicle index, but relies on the observations that (i) each route has a unique last customer and (ii) no customer can appear as last customer on more than one route.\footnote[2]{This formulation builds on the last-customer formulation for the min-max multiple traveling salesman, as presented by N. Bianchessi, C. Tilk, and S. Irnich at the 2023 International Workshop on Column Generation in Montréal.} We observe in preliminary experiments that the direct branch and price implementation of both formulations suffers from poor LP relaxations.

\paragraph{Vehicle-Index Formulation}
Let $R$ denote the index set of feasible routes. For each $r \in R$, let $c_r = \sum_{(i,j) \in r}c_{ij}$ and $p_r = \sum_{(i,j) \in r}p_{ij}$ denote the cost and the distance of route $r$, respectively, and let binary parameter $a_{ir}$ indicate whether customer $i$ is visited by route $r$. We introduce the binary decision variable $x_{rk}$, indicating whether or not route $r$ is assigned to vehicle $k$. Then, the following mixed-binary linear program is a valid formulation of F-CVRP:
\begin{subequations}
\begin{align}
\min \quad & \eta - \gamma \\
\text{s.t.} \quad  & \sum_{k \in [K]} \sum_{r \in R} a_{ir} x_{rk} = 1 && \forall i \in C \label{eq:fcvrp_vehicle_cover} \\
& \sum_{k \in [K]} \sum_{r \in R} c_r x_{rk} \leq B \\
& \sum_{r \in R} x_{rk} = 1 && \forall k \in [K] \\
& \sum_{r \in R} p_r x_{rk} \leq \eta && \forall k \in [K] \label{eq:fcvrp_vehicle_max} \\
& \sum_{r \in R} p_r x_{rk} \geq \gamma && \forall k \in [K] \label{eq:fcvrp_vehicle_min} \\
& x_{rk} \in \{0, 1\} && \forall k \in [K],~r \in R.
\end{align}
\label{eq:vehicle_index}
\end{subequations}
Here, constraints~(\ref{eq:fcvrp_vehicle_max})-(\ref{eq:fcvrp_vehicle_min}) are essentially special cases of constraints~(\ref{eq:general_min_max}), based on the fact that each vehicle always performs exactly one route. Due to the symmetry of the formulation in index $k$, this formulation always admits a fractional solution with objective value 0 (by setting $x_{r1} =x_{r2} =\ldots=x_{rK}$ for all $r \in R$). There exist techniques that partially break the symmetry among vehicles, but completely eliminating this symmetry in branch and price remains challenging \cite{darvish2020comparison}. 

\paragraph{Last-Customer Formulation}
We use the same parameters as in the vehicle-index formulation. In addition, for each route $r \in R$, we let the binary parameter $b_{ir}$ indicate whether customer $i$ is the last customer on route $r$. We introduce the binary decision variable $x_r$, indicating whether or not route $r$ is selected. Then, the following mixed-binary linear program is a valid formulation of F-CVRP:
\begin{subequations}
\begin{align}
\min \quad & \eta - \gamma \\
\text{s.t.} \quad  & \sum_{r \in R} a_{ir} x_{r}  = 1 && \forall i \in C \\
& \sum_{r \in R} c_r x_{r} \leq B \\
& \sum_{r \in R} x_{r} = K \\
& \sum_{r \in R} p_r b_{ir} x_r \leq \eta && \forall i \in C \label{eq:fcvrp_customer_max} \\
& M \left(1 - \sum_{r \in R} b_{ir} x_r \right) + \sum_{r \in R} p_r b_{ir} x_r \geq \gamma && \forall i \in C \label{eq:fcvrp_customer_min} \\
& x_{r} \in \{0, 1\} && \forall r \in R.
\end{align}
\label{eq:customer_formulation}
\end{subequations}
Constraints~(\ref{eq:fcvrp_customer_max})-(\ref{eq:fcvrp_customer_min}) are similar in spirit to constraints~(\ref{eq:general_min_max_weak}), based on the observation that each customer can appear as the last customer on at most one route. Big-$M$ constraints~(\ref{eq:fcvrp_customer_min}) in particular contribute to poor LP bounds. In addition, fractional solutions with low objective value can be obtained by selecting two copies of each route with equal value: one in forward orientation, and one in backward orientation (visiting all customers in reversed order). This issue, however, can be largely mitigated by enforcing that the index of the last customer exceeds that of the first customer within the pricing problem.

It is readily seen that both formulations are valid.
\begin{proposition}[Valid F-CVRP Formulations]
Formulations (\ref{eq:vehicle_index}) and (\ref{eq:customer_formulation}) are valid formulations of F-CVRP.
\end{proposition}

\subsection{Branch and Price}
\label{subsec:cvrp_branch}

In the following, we present a branch-and-price algorithm for solving the vehicle-index formulation, pointing out differences with the last-customer formulation whenever applicable. We refer to \cite{costa2019exact} for a detailed introduction to branch and price for vehicle routing problems.

\paragraph{Pricing Problem}
Let $\bm{\kappa}, \lambda, \bm{\mu}, \bm{\nu}$, and $\bm{\pi}$ denote the dual variables of constraints (\ref{eq:fcvrp_vehicle_cover})-(\ref{eq:fcvrp_vehicle_min}), respectively, of the vehicle-index formulation \eqref{eq:vehicle_index}. The reduced cost of route $r$ of vehicle $k$ is then given by
\begin{subequations}
\begin{align}
   RC(x_{rk}) =& -\mu_k -\sum_{i \in C} \kappa_{i} a_{ir} - \lambda c_r - (\nu_k + \pi_k) p_r \\
   =&-\mu_k - \sum_{(i, j) \in r} (\kappa_{j} + \lambda c_{ij} + (\nu_k + \pi_k) p_{ij}).
\end{align}
\end{subequations}
The pricing problem consists of finding a feasible route with negative reduced cost. The reduced cost can be easily decomposed on the arcs of $G$, and hence the pricing problem can be formulated as a series of elementary shortest path problems with resource constraints. There is a single pricing problem per vehicle, and resource constraints model the vehicle capacity constraint. We can partly break symmetry by enforcing that customer $i$ is served by a vehicle with index $k \leq i$, i.e., by removing customer $i$ from all pricing problems  where $k > i$. In the last-customer formulation, there is a single pricing problem per customer. 

\paragraph{Pricing Algorithm}

As typically done in branch-and-price algorithms for vehicle routing problems, we solve the pricing problem using a labeling algorithm \cite{costa2019exact}. Each label represents a partial path from the depot. The algorithm is initialized with a single dummy label at the depot. For each remaining label, its partial path is extended in all possible directions. For each extended label, we check whether it corresponds to a potentially feasible route completion. To reduce the number of labels, we check whether an extended label is dominated by some existing label, or whether it dominates an existing label. Dominated labels are removed from consideration. Labels returning to the depot with negative reduced cost are stored in memory. We now discuss the different components of the labeling algorithm in more detail.

A label $L = (v(L), c(L), q(L), \Pi(L))$ is a tuple representing a partial path from the depot to customer $v(L)$, with reduced cost $c(L)$, load $q(L)$, and set of previously visited customers $\Pi(L)$. In practice, each label also stores its predecessor to enable backtracking. It is feasible to extend label $L$ along arc $(v(L), j)$ if $j \notin \Pi(L)$ and $q(L) + d_j \leq Q$, i.e., when customer $j$ has not yet been visited and there is sufficient remaining capacity to cover demand $d_j$. In this case, the extended label $L'$ is given by $L' = (j, c(L) - (\kappa_j + \lambda c_{ij} + (\nu_k + \pi_k) p_{ij}), q(L) + d_j, \Pi(L) \cup \{ j\})$. A label corresponds to a route with negative reduced cost when $v(L) =0$ and $c(L) < \mu_k$. 

Dominance rules are applied to reduce the number of labels to consider. Label $L_1$ is said to dominate $L_2$ when any feasible extension of $L_2$ can be used to extend $L_1$ to a feasible route with lower or equal reduced cost. The following dominance checks form a set of sufficient conditions:
\begin{equation}
    v(L_1) = v(L_2)~\land~q(L_1) \leq q(L_2) ~\land~ c(L_1) \leq c(L_2) ~\land~ \Pi(L_1) \subseteq \Pi(L_2) .
\end{equation}
We use several well-known acceleration techniques to speed-up the labeling algorithm. First, to allow for efficient dominance checks, we store labels in buckets based on their final customer and demand, and only perform intra-bucket dominance checks. Second, we use bi-directional labeling, i.e., extend partial paths both starting from and ending at the depot. Finally, we use $ng$-path relaxation, where the customer memory $\Pi$ is truncated to only store nodes in a local neighborhood around the current node \cite{baldacci2011new}. While this allows for non-elementary routes, elementarity is restored through branching and the number of non-dominated labels is significantly reduced. We refer to \cite{costa2019exact} for more information on these acceleration techniques. 

\paragraph{Branching Rules} 
For the vehicle-index formulation \eqref{eq:vehicle_index}, we first branch on the most fractional assignment of a customer to a vehicle, i.e., on the value $\sum_{r \in R} a_{ir} x_{rk}$. When a customer is assigned to a vehicle, it is removed from the pricing problem of all other vehicles. Not assigning a customer to a vehicle is processed similarly. In case of the last-customer formulation, we instead branch on whether a customer is last on a route or not. Second, we branch on the most fractional arc. A forbidden arc is removed from all pricing problems. Arc $(i, j)$ can be forced in a solution by removing all other arcs leaving customer $i$ or entering customer $j$, with special attention paid to the case where $i$ or $j$ equals the depot. Our range branching rule acts on top of these problem-specific branching rules, i.e., the problem-specific branching rules are only invoked when range branching does not detect any branching opportunities.

\paragraph{Range Branching}

Under the range branching scheme, we need to be able to prevent the generation of routes with distance below $L$ or above $U$ in the pricing problem. As such, we must extend each label $L$ to also store the distance $d(L)$ of the partial path. To check whether label $L_1$ dominates label $L_2$, one of the following additional dominance rules must be satisfied:
\begin{align}
\begin{split}
    \left(L = 0 ~\land~ U < \infty ~\land~ d(L_1) \leq d(L_2) \right) ~\lor~ & \left(L > 0 ~\land~ U = \infty ~\land~ d(L_1) \geq d(L_2) \right) \\ ~\lor~ & \left(L > 0 ~\land~ U < \infty ~\land~ d(L_1) = d(L_2) \right).
\end{split}
\end{align}
These additional dominance rules can significantly slow down the labeling algorithm, as they potentially cause a strong increase in the number of non-dominated labels.

\subsection{Set-Up}
\label{subsec:cvrp_setup}

Similar to \cite{matl2019workload}, we construct instances based on CVRP instance \texttt{X64} \cite{uchoa2017new}. For each number $|C| \in \{15, 20, 25\}$, we construct 20 instances by selecting $|C| + 1$ customer locations from the original instance, where the first customer location serves as depot. Each instance has $K = 5$ vehicles with capacity $Q$ chosen such that all vehicles are needed to serve the demand. We set the budget $B$ equal to $110\%$ of the cost of the cost-efficient solution. The payoff of a route is equal to its distance, i.e., we aim to balance the length of the routes.

We consider the vehicle-index and last-customer formulations for the F-CVRP presented in Section~\ref{subsec:cvrp_math}, and solve both formulations using both the classical and the range branching scheme. We provide the range of the efficient solution as an upper bound to the branch-and-price algorithm, impose a time limit of one hour, and solve up to 8 pricing problems in parallel. We return at most 20 columns per pricing problem, and remove columns that have been inactive for 20 iterations. Based on preliminary computational experiments, we implement range branching with cutoff values defined by $\alpha = 0.025$. All linear programs are solved using \texttt{CPLEX 22.1.0}.

\subsection{Results}
\label{subsec:cvrp_results}

Table~\ref{tab:results} presents the results for the three sets of instances. For each combination of number of customers, formulation, and branching scheme, we present the number of instances solved to optimality within the time limit, the computing time, the optimality gap, the number of branch-and-bound nodes, and the percentage-wise reduction in range of the best solution found compared to the most cost-efficient solution ($\Delta)$. For each number of customers, the results are averaged over all 20 instances. In case not all instances are solved, the computing times and number of branch-and-bound nodes provide lower bounds on the true values. 

\begin{table}[tb!]
\caption{Computational performance of vehicle-index and last-customer formulations with classical and range branching scheme.}
\label{tab:results}
\centering 
\renewcommand{\arraystretch}{1}
\begin{tabular}{@{\extracolsep{5pt}}cllrrrrr}
\toprule
$|C|$ & Formulation & Branching & \# Solved & Time (s) & Gap (\%) & \# Nodes & $\Delta$ (\%) \\ 
\midrule
\multirow{4}{*}{15} & \multirow{2}{*}{Vehicle} & Classical & 20/20 & 207.1 & 0 & 6,619.4 & 49.4 \\ 
& & Range & 20/20 & 175.4 & 0 & 5,687.8 & 49.4 \\ 
& \multirow{2}{*}{Customer} & Classical & 17/20 & 916.9 & 4.0 & 15,726.8 & 48.7 \\ 
& & Range & 20/20 & 3.6 & 0 & 385.8 & 49.4 \\ 
\midrule
\multirow{4}{*}{20} & \multirow{2}{*}{Vehicle} & Classical & 9/20 & 2,769.5 & 33.3 & 14,446.9 & 34.2 \\ 
& & Range & 16/20 & 1,099.0 & 12.3 & 6,386.6 & 44.8 \\ 
& \multirow{2}{*}{Customer} & Classical & 1/20 & 3,494.3 & 67.4 & 21,279.1 & 25.4 \\ 
& & Range & 19/20 & 290.6 & 0.0 & 3,430.9 & 53.5 \\ 
\midrule 
\multirow{4}{*}{25} & \multirow{2}{*}{Vehicle} & Classical & 0/20 & 3,600.0 & 94.0 & 11,349.4 & 8.6 \\ 
& & Range & 6/20 & 3,053.2 & 25.2 & 983.8 & 38.6 \\ 
& \multirow{2}{*}{Customer} & Classical & 0/20 & 3,600.0 & 94.1 & 13,220.6 & 5.1 \\ 
& & Range & 14/20 & 1,741.1 & 5.8 & 1,728.1 & 54.0 \\ 
\bottomrule
\end{tabular}
\end{table}

The results in Table~\ref{tab:results} show that range branching outperforms classical branching in all cases. For all sizes and formulations, range branching leads to a strong reduction in computing time, optimality gap, and number of branch-and-bound nodes. In addition, it allows us to solve 14 more instances with $|C|=25$ customers, the largest instances we consider. The last-customer formulation seems to particularly benefit from range branching, an effect that appears to be largely driven by a huge reduction in the size of the branch-and-bound tree. While it is not competitive with the vehicle-index formulation under the classical branching scheme, it dominates this formulation on all instance sets when range branching is used. Finally, we note that fairness gains up to 54\% can be attained compared to the most cost-efficient solution.

\begin{figure}[htb!]
\centering
\begin{tikzpicture}[scale=0.9]
    \begin{groupplot}[group style={
                      group name=myplot,
                      group size= 2 by 3},
                      height=6cm,
                      width=6.5cm]
    \pgfplotsset{group/every plot/.append style={
        ymin=0, ymax=3, xmin=0, xmax=3600, extra x ticks={0}, extra y ticks={0}, grid, grid style={opacity = 0.5, dotted, black}, disabledatascaling, xtick={900, 1800, 2700, 3600}, ytick={0.5, 1.0, 1.5, 2.0, 2.5, 3.0}}}
        
    \nextgroupplot[title={Vehicle},ylabel={$N=15$}]
    \addplot[very thick, color=black!45] table[x index=0, y index=1, col sep=semicolon] {Results/bounds_15_0.csv};\label{plots:plot1}

    \addplot[very thick, dashed, color=black!45]
     table[x index=0, y index=2, col sep=semicolon] {Results/bounds_15_0.csv}; \label{plots:plot2}

    \addplot[very thick, color=black] table[x index=0, y index=1, col sep=semicolon] {Results/bounds_15_1.csv};\label{plots:plot3}

    \addplot[very thick, dashed, color=black] table[x index=0, y index=2, col sep=semicolon] {Results/bounds_15_1.csv};\label{plots:plot4}
     
    \nextgroupplot[title={Customer}]
    \addplot[very thick, color=black!45] table[x index=0, y index=1, col sep=semicolon] {Results/bounds_15_2.csv};

    \addplot[very thick, dashed, color=black!45]
     table[x index=0, y index=2, col sep=semicolon] {Results/bounds_15_2.csv};

    \addplot[very thick, color=black] table[x index=0, y index=1, col sep=semicolon] {Results/bounds_15_3.csv};

    \addplot[very thick, dashed, color=black] table[x index=0, y index=2, col sep=semicolon] {Results/bounds_15_3.csv};
        \nextgroupplot[ylabel={$N=20$}]
    \addplot[very thick, color=black!45] table[x index=0, y index=1, col sep=semicolon] {Results/bounds_20_0.csv};

    \addplot[very thick, dashed, color=black!45]
     table[x index=0, y index=2, col sep=semicolon] {Results/bounds_20_0.csv};

    \addplot[very thick, color=black] table[x index=0, y index=1, col sep=semicolon] {Results/bounds_20_1.csv};

    \addplot[very thick, dashed, color=black] table[x index=0, y index=2, col sep=semicolon] {Results/bounds_20_1.csv};
        \nextgroupplot
    \addplot[very thick, color=black!45] table[x index=0, y index=1, col sep=semicolon] {Results/bounds_20_2.csv};

    \addplot[very thick, dashed, color=black!45]
     table[x index=0, y index=2, col sep=semicolon] {Results/bounds_20_2.csv};

    \addplot[very thick, color=black] table[x index=0, y index=1, col sep=semicolon] {Results/bounds_20_3.csv};

    \addplot[very thick, dashed, color=black] table[x index=0, y index=2, col sep=semicolon] {Results/bounds_20_3.csv};
    \nextgroupplot[xlabel={Time (s)}, ylabel={$N=25$}]
    \addplot[very thick, color=black!45] table[x index=0, y index=1, col sep=semicolon] {Results/bounds_25_0.csv};

    \addplot[very thick, dashed, color=black!45]
     table[x index=0, y index=2, col sep=semicolon] {Results/bounds_25_0.csv};
     
    \addplot[very thick, color=black] table[x index=0, y index=1, col sep=semicolon] {Results/bounds_25_1.csv};

    \addplot[very thick, dashed, color=black] table[x index=0, y index=2, col sep=semicolon] {Results/bounds_25_1.csv};
    \nextgroupplot[xlabel={Time (s)}]
    \addplot[very thick, color=black!45] table[x index=0, y index=1, col sep=semicolon] {Results/bounds_25_2.csv};

    \addplot[very thick, dashed, color=black!45]
     table[x index=0, y index=2, col sep=semicolon] {Results/bounds_25_2.csv};
     
    \addplot[very thick, color=black] table[x index=0, y index=1, col sep=semicolon] {Results/bounds_25_3.csv};

    \addplot[very thick, dashed, color=black] table[x index=0, y index=2, col sep=semicolon] {Results/bounds_25_3.csv};
    \end{groupplot}
    \path (myplot c1r1.outer north west)
          -- node[anchor=south,rotate=90] {Fraction of best known bound}
          (myplot c1r3.outer south west);
\path (myplot c1r1.north west|-current bounding box.north)--
      coordinate(legendpos)
      (myplot c2r1.north east|-current bounding box.north);
\matrix[
    matrix of nodes,
    anchor=south,
    draw,
    inner sep=0.2em,
    draw
  ]at([yshift=1ex]legendpos)
  {\ref{plots:plot1}& Classical (LB)&[5pt]
\ref{plots:plot2}& Classical (UB)\\
\ref{plots:plot3}& Range (LB)&[5pt]
\ref{plots:plot4}& Range (UB)\\};
\end{tikzpicture}
\caption{Progression of lower and upper bound of vehicle-index formulation (left) and last-customer formulation (right) with classical and range branching scheme.}
\label{fig:results_bounds}
\end{figure}
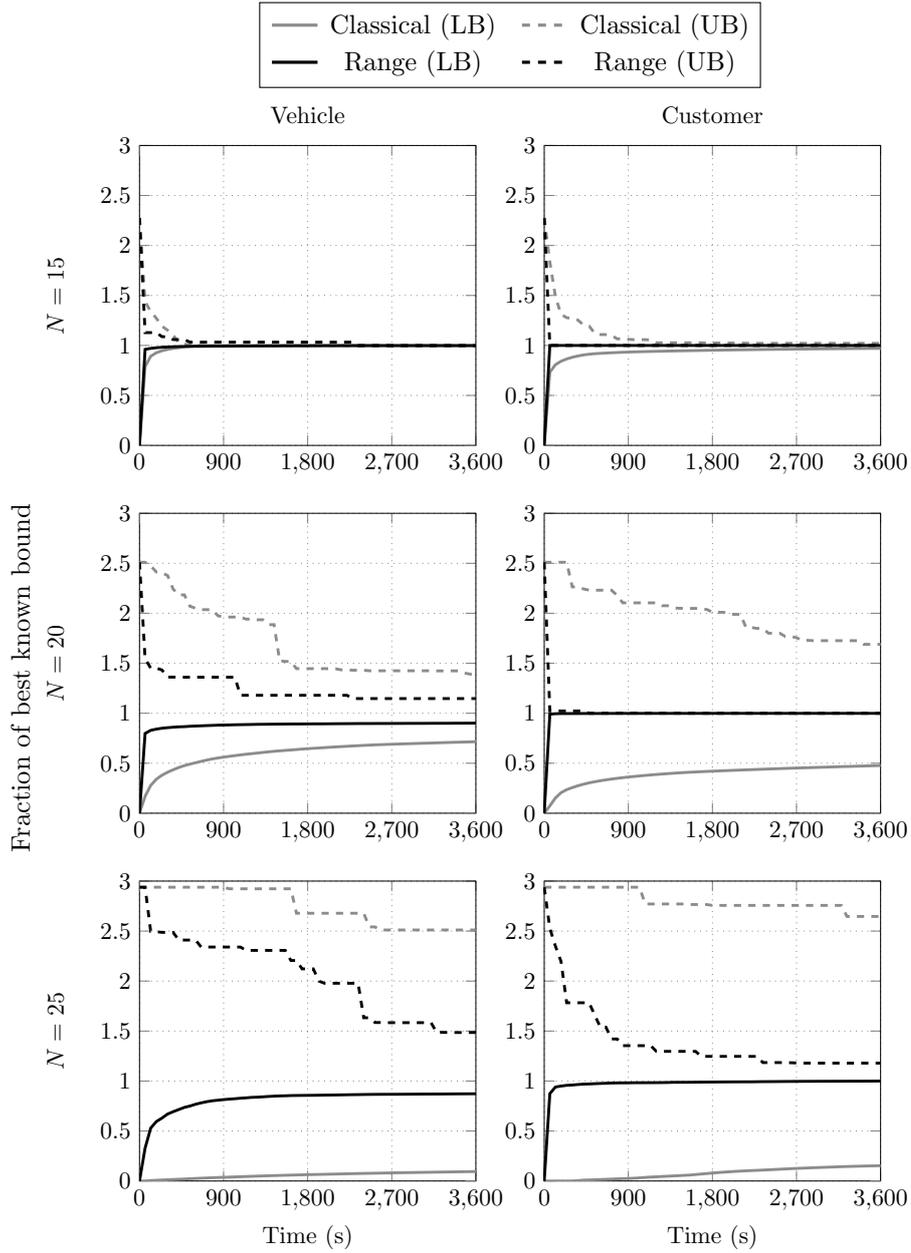

Figure~\ref{fig:results_bounds} displays the progression of the lower and upper bounds, for each combination of number of customers, formulation, and branching scheme, throughout the course of the branch-and-price algorithm. These relative values are expressed in terms of the best known lower bound, and subsequently averaged over all instances. We find that the performance of range branching is largely driven by its ability to quickly drive up the lower bound to near its final value. This is well in line with the theoretical properties of range branching derived in Section~\ref{subsec:branching_theory}. On the larger instances, the lower bounds of the classical branching scheme are still slowly increasing after an hour of computation, whereas the bounds with range branching reach near-optimal values early in the process. This is especially true for the last-customer formulation that attains near-optimal lower bounds in the first few iterations of the algorithm. This might explain why it outperforms the vehicle-index formulation. In addition to improved lower bounds, range branching also has a positive effect on the upper bound evolution throughout the course of branch and price. Hence, range branching appears to be the method of choice even on large instances that are not solved to optimality within the time limit.

\subsection{TSP-Optimal Routes}
\label{subsec:cvrp_tsp}

In our setup, the range of a solution can potentially be reduced by (i) exchanging customers across routes or (ii) changing the customer order along a route. The latter can be undesirable in practice, as the resulting routes might violate TSP-optimality. That is, the routes might deviate from the shortest tour visiting all customers and the depot, resulting in an unnecessarily large distance to be traveled. In this section, we conduct additional experiments to show that directly enforcing TSP-optimality is computationally demanding, but a small postprocessing step is sufficient for our method to perform well on the problem restricted to TSP-optimal routes.

We start by comparing the performance of range branching on the standard F-CVRP and the F-CVRP restricted to TSP-optimal routes. Since the impact of range branching for the TSP-optimal case is similar to the general case of Section~\ref{subsec:cvrp_results}\footnote{We compared the classical B\&P algorithm with range-B\&P on instances featuring $|C|=15$ customers and a route budget equal to $B=101\%$ of the cost-efficient solution. The classical algorithm solved 5 out of 20 instances to optimality and attained an average optimality gap of 17.2\%. Range branching solved 17 out of 20 instances, attaining an average gap of 0.9\% (see the first row of Table~\ref{tab:tsp_computational}).}, we do not consider the classical branching scheme here. The first algorithm, denoted by `General', is the range-B\&P algorithm applied to the last-customer formulation. This algorithm allows TSP-violating routes. The second algorithm, denoted by `TSP', is range-B\&P applied to the vehicle-index formulation. Here, we modify the dominance rules to ensure that only TSP-optimal routes are generated by the pricing algorithm. In particular, one partial route dominates another when it covers the same customers in a shorter or equal distance. Note that this introduces a significant weakening of the dominance rules. We compare the algorithms using the same instances and parameter settings as in Section~\ref{subsec:cvrp_setup}. Since the degree to which routes violate TSP-optimality is likely to depend on the total route budget, we vary the budget $B$ to be 101\%, 105\%, and 110\% of the cost-efficient solution.

Table~\ref{tab:tsp_computational} presents the computational results of these experiments. For each combination of number of customers, route budget, and algorithm, we present the number of instances solved to optimality within the time limit, the computing time, and the optimality gap. For each combination of number of customers and budget, the results are averaged over all 20 instances. In case not all instances are solved, the computing times and number of branch-and-bound nodes provide lower bounds on the true values. A dash indicates that the root node relaxation could not be solved to optimality within the time limit.

\begin{table}[tb!]
\caption{Computational performance of range branching with only TSP-optimal routes and with general routes.}
\label{tab:tsp_computational}
\centering 
\renewcommand{\arraystretch}{1}
\begin{tabular}{@{\extracolsep{5pt}}clrrrrrr}
\toprule
& & \multicolumn{3}{c}{General} & \multicolumn{3}{c}{TSP} \\
\cmidrule(l){3-5} \cmidrule(l){6-8} 
$|C|$ & $B$ (\%) & \# Solved & Time (s) & Gap (\%) & \# Solved & Time (s) & Gap (\%)\\ 
\midrule
\multirow{3}{*}{15} & 101\% & 20/20 & 8.0 & 0 & 17/20 & 824.4 & 0.9 \\ 
& 105\% & 20/20 & 70.6 & 0 & 16/20 & 956.9 & 3.6 \\ 
& 110\% & 20/20 & 3.6 & 0 & 20/20 & 476.6 & 0 \\ 
\midrule
\multirow{3}{*}{20} & 101\% & 18/20 & 576.9 & 1.1 & 3/20 & 3,479.4 & 10.0 \\ 
& 105\% & 19/20 & 294.0 & 0.0 & 0/20 & 3,600.0 & 27.5 \\ 
& 110\% & 19/20 & 290.6 & 0.0 & 0/20 & 3,600.0 & 41.5 \\ 
\midrule
\multirow{3}{*}{25} & 101\% & 13/20 & 1,865.1 & 7.1 & - & - & - \\ 
& 105\% & 12/20 & 1,841.4 & 14.2  & - & - & - \\ 
& 110\% & 14/20 & 1,741.1 & 5.8 & - & - & - \\ 
\bottomrule
\end{tabular}
\end{table}

The results in Table~\ref{tab:tsp_computational} show the additional computational complexity that arises when enforcing TSP-optimality. The performance of the range-B\&P algorithm with general routes is in line with the results of Section~\ref{subsec:cvrp_results}. A majority of the instances is solved to optimality, and the route budget does not appear to have any significant effect on computational performance. However, enforcing TSP-optimal routes comes at a severe computational cost. Only a handful of 20-customer instances is solved to optimality, and all of the 25-customer instances exceed the time limit before solving the root node. This is not surprising, given that the modified dominance rules imply a near-enumeration of all possible routes by the labeling algorithm. We conclude that directly enforcing TSP-optimality is intractable for all non-trivial instance sizes.

Fortunately, the range-B\&P algorithm with general routes provides an alternative way of solving the F-CVRP with only TSP-optimal routes. In particular, any lower bound for the problem with general routes is a valid lower bound for the problem with only TSP-optimal routes. Second, a solution containing non-TSP routes can be easily converted to a feasible one by rearranging the order of customers in each route, for example, by using a compact MILP formulation of the TSP. This yields a valid upper bound.

Table~\ref{tab:tsp_bounds} compares the bounds obtained in this way with the bounds from the TSP-algorithm. For all number of customers and budgets, it reports the number of instances for which the converted solutions are provably optimal, the relative percentage by which the range was underestimated by allowing non-TSP routes ($\Delta_R$), the relative percentage increase in lower bound compared to the TSP-algorithm ($\Delta_{LB}$), and the relative percentage reduction in upper bound compared to the TSP-algorithm ($\Delta_{UB}$). Finally, it reports the optimality gap of the converted solutions, computed using the best-known lower bound from both methods. A dash indicates that the TSP-algorithm did not compute any valid bounds within the time limit.

\begin{table}[tb!]
\caption{Analysis of bounds obtained by converting general routes to TSP-optimal routes.}
\label{tab:tsp_bounds}
\centering 
\renewcommand{\arraystretch}{1}
\begin{tabular}{@{\extracolsep{5pt}}clrrrrr}
\toprule
$|C|$ & $B$ (\%) & \# TSP-Optimal & $\Delta_R$ (\%) & $\Delta_{LB}$ (\%) & $\Delta_{UB}$ (\%) & Gap (\%) \\ 
\midrule
\multirow{3}{*}{15} & 101\% & 19/20 & 0.0 & 0.3 & 0.6 & 0.0 \\ 
& 105\% & 18/20 & 1.2 & 1.4 & 1.7 & 0.3 \\ 
& 110\% & 13/20 & 7.7 & -4.8 & -3.8 & 3.2 \\ 
\midrule
\multirow{3}{*}{20} & 101\% & 17/20 & 0.3 & 6.6 & 4.1 & 1.2 \\ 
& 105\% & 8/20 & 4.1 & 5.9 & 20.7 & 3.1 \\ 
& 110\% & 4/20 & 13.1 & 1.3 & 28.1 & 10.6 \\ 
\midrule
\multirow{3}{*}{25} & 101\% & 9/20 & 0.4 & - & - & 7.6 \\ 
& 105\% & 4/20 & 2.7 & - & - & 16.9 \\ 
& 110\% & 3/20 & 14.9 & - & - & 21.3 \\ 
\bottomrule
\end{tabular}
\end{table}

The results of Table~\ref{tab:tsp_bounds} show that the proposed method forms an effective heuristic for the F-CVRP with only TSP-optimal routes, and even solves many of the problem instances to optimality. It provides improved bounds, indicated by positive values of $\Delta_{LB}$ and $\Delta_{UB}$, on nearly all instances. The only exception is the case with 15 customers and a large route budget of 110\%, in which (i) the instance is sufficiently small as to be solved to optimality by the TSP-algorithm and (ii) the budget is sufficiently large as to allow for many TSP-violated routes in the general model. The increase in range after converting routes to TSP-optimal ones, i.e., the degree to which TSP-violating routes are used, is generally negligible for smaller route budgets, but becomes sizable as the route budget reaches 110\% of the cost-efficient solution. Similarly, the optimality gaps are well below 10\% except for the larger instances with high route budgets. Given the inefficiency of the modified labeling algorithm, we conclude that it is preferable to use general routes and convert them to TSP-optimal routes in a post-processing step, than to directly enforce TSP-optimality in the pricing problem. Of course, the possibility of devising an entirely new algorithm for the TSP-optimal case remains open.

\section{Fair Generalized Assignment Problem}
\label{sec:gap}

In this section, we introduce an application of the range minimization problem that we call the \emph{fair generalized assignment problem (F-GAP)}, a fairness-oriented version of the generalized assignment problem.

\begin{definition}[Fair Generalized Assignment Problem]
Consider $m$ jobs to be assigned to $n$ agents. Assigning job $j \in [m]$ to agent $i \in [n]$ yields a profit $p_{ij}$ and consumes $c_{ij}$ units of the resource of agent $i$. The resource capacity of agent $i$ equals $C_i$. All costs and capacities are integer-valued. An assignment is feasible when the total resource consumption of each agent does not exceed his capacity, and the overall profit of all jobs is at least $P$. The goal of the fair generalized assignment problem is to find a feasible assignment that minimizes the range of the resource consumption of different agents.
\end{definition}

Branch and price is not necessarily the preferred solution method for the generalized assignment problem, as commercial solvers are typically faster in solving the regular, cost-oriented problem. We show that, despite this fact, branch and price is several factors faster than solving a standard compact formulation for F-GAP. In other words, when minimizing range it can be beneficial to switch to a Dantzig-Wolfe reformulation as compared to a standard compact formulation.

We present two mathematical formulations of F-GAP in Section~\ref{subsec:gap_math} and present a branch-and-price algorithm with range branching in Section~\ref{subsec:gap_branch}. We describe the set-up of our experiments in Section~\ref{subsec:gap_setup} and discuss the results in Section~\ref{subsec:gap_results}.

\subsection{Mathematical Formulation}
\label{subsec:gap_math}

\paragraph{Compact Formulation}

We introduce the binary decision variable $x_{ij}$ to indicate whether or not job $j$ is assigned to agent $i$. A compact formulation of F-GAP reads as
\begin{subequations}
\begin{align}
\min \quad & \eta - \gamma & \\ 
\text{s.t.} \quad & \sum_{i \in [n]} \sum_{j \in [m]} p_{ij} x_{ij} \geq P \\
& \sum_{i \in [n]} x_{ij} = 1 && \forall j \in [m] \\
& \sum_{j \in [m]} c_{ij} x_{ij} \leq C_{i} && \forall i \in [n] \\ 
& \sum_{j \in [m]} c_{ij} x_{ij} \leq \eta && \forall i \in [n] \\ 
& \sum_{j \in [m]} c_{ij} x_{ij} \geq \gamma && \forall i \in [n] \\
& x_{ij} \in \{0, 1\} && \forall i \in [n],~j \in [m].
\end{align}
\label{eq:fgap_compact}
\end{subequations}

\paragraph{Dantzig-Wolfe Reformulation}

We define an assignment as a set of jobs that can be performed by a single agent without violating its resource constraint. Formally, an assignment $a=(i, S)$ is a tuple denoting that agent $i$ performs all the jobs in set $S\subseteq\{1,\ldots,m\}$ with $\sum_{j\in S}c_{ij}\leq C_i$. Let $A$ be the set of feasible assignments, with $A_{i} \subseteq A$ denoting the set of feasible assignments to agent $i$. For each $a \in A$, let $p_{a}$ denote the profit of assignment $a$, $c_a$ the resource consumption of assignment $a$, and let $k_{aj}$ be a binary parameter indicating whether job $j$ is included in assignment $a$. We introduce the binary decision variable $x_{a}$ indicating whether or not assignment $a$ is selected. A Dantzig-Wolfe reformulation of F-GAP then reads as
\begin{subequations}
\begin{align}
\min \quad & \eta - \gamma \label{eq:fgap_obj} \\ 
\text{s.t.}\quad & \sum_{a \in A} p_{a} x_{a} \geq P \label{eq:fgap_profit} \\
& \sum_{a \in A} k_{aj} x_a=1&& \forall j \in [m] \label{eq:fgap_cover} \\
& \sum_{a \in A_{i}} x_a \leq 1 && \forall i \in [n] \label{eq:fgap_assignment} \\
& \sum_{a \in A_{i}} c_a x_a \leq \eta && \forall i \in [n] \label{eq:fgap_max} \\
& \sum_{a \in A_{i}} c_a x_a \geq \gamma && \forall i \in [n] \label{eq:fgap_min} \\
& x_{a} \in \{0, 1\} && \forall a \in A. \label{eq:fgap_domain}
\end{align}
\label{eq:fgap}
\end{subequations}
Note that constraints~(\ref{eq:fgap_max})-(\ref{eq:fgap_min}) are formulated in line with general constraints~(\ref{eq:general_min_max}), following from the fact that each agent is always assigned exactly one assignment. The reformulation contains an exponential number of variables, but can be solved efficiently by branch and price \cite{savelsbergh1997branch}. It is easy to see that the reformulation is a valid formulation of F-GAP.

\begin{proposition}[Valid F-GAP Formulation]
Formulation (\ref{eq:fgap}) is a valid formulation of F-GAP.
\end{proposition}

Recall that the concept of a valid formulation does not apply to the compact formulation (\ref{eq:fgap_compact}), since it is not a column-based formulation and therefore does not meet the conditions of Definition~\ref{def:range_problem}.

\subsection{Branch and Price}
\label{subsec:gap_branch}

We largely follow the branch-and-price implementation of \cite{savelsbergh1997branch}, and show how to augment it with our range branching rule.

\paragraph{Pricing Problem and Algorithm}

Let $\kappa, \bm{\lambda}, \bm{\mu}, \bm{\nu}$, and $\bm{\pi}$ denote the dual variables of constraints (\ref{eq:fgap_profit})-(\ref{eq:fgap_min}), respectively. The reduced cost of assignment $a \in A_i$ is then given by
\begin{align}
   RC(x_a) =& -\mu_i - \kappa p_a - \sum_{j \in [m]} \lambda_j k_{aj} - c_a(\nu_i + \pi_i)
   \\
   =& -\mu_i - \sum_{j \in [m]} (\lambda_j + \kappa p_{ij} + (\nu_i + \pi) c_{ij}) k_{aj}.
\end{align}
The pricing problem consists of finding a feasible assignment with a negative reduced cost.  It can be modeled as a series of knapsack problems, one per agent. Here, the knapsack problem associated with agent $i$ has capacity $C_i$, and job $j$ has weight $c_{ij}$ and value $\lambda_j + \kappa p_{ij} + (\nu_i + \pi) c_{ij}$. We solve each knapsack problem using a custom implementation of the well-known pseudo-polynomial dynamic programming (DP) algorithm. A DP table of size $\mathcal{O}(m C_i)$ is constructed, where entry $(k, c)$ corresponds to the maximum value that can be attained by filling a knapsack of size $c$ with items $\{1, \dots, k\}$. The value of the optimal knapsack is given by entry $(m, C_i)$. 

\paragraph{Branching Rule}

We branch on the most fractional assignment of a job to an agent, i.e., on the value of $\sum_{a \in A_i} k_{aj} x_a$. Note that this corresponds to branching on $x_{ij}$ in the compact formulation. If a job is assigned to an agent, all assignments of this agent not containing the job are removed from the model and the remaining resource capacity in the pricing problem is adjusted downwards. Assignments to other agents containing this job are removed from the problem and the job is removed from the pricing problem of these agents. Not assigning a job to an agent is processed similarly. Our range branching rule acts on top of this problem-specific branching rule and is always invoked first.

\paragraph{Range Branching}

To adhere to the variable fixing strategy of the range branching rule, we need to be able to prevent the generation of assignments with resource consumption above $U$ or below $L$ in the pricing problem. The former is readily achieved by artificially reducing the knapsack capacity from $C_i$ to $U$. The latter requires a slight adjustment of the DP table. We re-define entry $(k, c)$ to be equal to the maximum value of a knapsack of weight \emph{exactly} $c$ with items $\{1, \dots, k\}$. These entries can be computed with the same recursive relation as in the original dynamic program after a minor change in the initialisation procedure. The value of the optimal knapsack is obtained by taking the maximum over entries from $(m, L)$ up to $(m, U)$.

\subsection{Set-Up}
\label{subsec:gap_setup}

We evaluate the performance of range branching by comparing it with a classical branching scheme as well as a MILP-solver on the compact formulation. We use benchmark instances for the generalized assignment problem proposed by \cite{chu1997genetic}. In particular, we consider the instances in classes A, B, and C with 100 jobs and 5, 10, and 20 agents, respectively. We compute the profit-optimal solution value $P^*$ to these instances and require a minimum profit of $(1 - \theta)P^*$, where $\theta$ ranges from 0.25\% up to 1\% in steps of 0.25\%. We impose a time limit of one hour, return at most one column per pricing problem, and remove columns that have been inactive for 30 iterations. We use range-branching with cutoff values corresponding to $\alpha=0$, i.e., we branch on all solutions that are not strictly range-respecting. All linear programs, as well as the compact formulation, are solved with \texttt{CPLEX 22.1.0}.

\subsection{Results}
\label{subsec:gap_results}

\begin{table}[htb!]
\caption{Computational performance of compact formulation and Dantzig-Wolfe reformulation with classical and range branching scheme.}
\label{tab:results_gap}
\centering
\renewcommand{\arraystretch}{1}
\begin{tabular}{@{\extracolsep{5pt}}cllrrrr} 
\toprule
\# Agents & $1 - \theta$ & Method & \# Solved & Time (s) & Gap (\%) & $\Delta$ (\%) \\ 
\midrule
\multirow{12}{*}{5} & \multirow{3}{*}{0.99}& MILP & 3/3 & 117.0 & 0 & 99.0 \\
 & & B\&P & 1/3 & 2,446.3 & 66.7 & 0 \\
 & & Range-B\&P & 3/3 & 166.9 & 0 & 99.0 \\
 \cmidrule{2-7}
& \multirow{3}{*}{0.9925} & MILP & 3/3 & 85.4 & 0 & 96.1 \\
 & & B\&P & 1/3 & 2,487.2 & 63.8 & 0\\
 & & Range-B\&P & 3/3 & 144.8 & 0 & 96.1 \\
 \cmidrule{2-7}
& \multirow{3}{*}{0.995} & MILP  & 3/3 & 67.8 & 0 & 91.0 \\
 & & B\&P & 2/3 & 1,930.5 & 33.3 & 57.8 \\
 & & Range-B\&P & 3/3 & 76.3 & 0 & 91.0 \\
 \cmidrule{2-7}
& \multirow{3}{*}{0.9975} & MILP & 3/3 & 34.6 & 0 & 70.1 \\
 & & B\&P & 3/3 & 974.6  & 0 & 70.1 \\
 & & Range-B\&P & 3/3 & 106.2 & 0 & 70.1 \\
\midrule
\multirow{12}{*}{10}  & \multirow{3}{*}{0.99}& MILP & 1/3 & 2,716.4 & 66.7 & 91.7 \\
 & & B\&P & 0/3 & 3,600.0 & 100.00 & 1.9 \\
 & & Range-B\&P & 3/3 & 52.0 & 0 & 99.7 \\
 \cmidrule{2-7}
& \multirow{3}{*}{0.9925} & MILP & 2/3 & 1,325.2 & 33.3 & 93.4 \\
 & & B\&P & 0/3 & 3,600.0 & 100.0 & 1.9 \\
 & & Range-B\&P & 3/3 & 90.5 & 0 & 97.1 \\
 \cmidrule{2-7}
& \multirow{3}{*}{0.995} & MILP & 2/3 & 1,262.4 & 33.3 & 85.5 \\
 & & B\&P  & 0/3 & 3,600.0 & 94.0 & 1.9 \\
 & & Range-B\&P & 3/3 & 132.8 & 0 & 89.2 \\
 \cmidrule{2-7}
& \multirow{3}{*}{0.9975} & MILP & 3/3 & 113.2 & 0 & 72.5 \\
 & & B\&P & 0/3 & 3,600.0 & 77.6 & 1.9 \\
 & & Range-B\&P & 3/3 & 22.8 & 0 & 72.5 \\
 \midrule
\multirow{12}{*}{20} & \multirow{3}{*}{0.99}& MILP & 0/3 & 3,600.0 & 100.0 & 88.9 \\
 & & B\&P & 0/3 & 3,600.0 & 100.0 & 0 \\
 & & Range-B\&P & 3/3 & 758.6 & 0 & 96.2 \\
 \cmidrule{2-7}
& \multirow{3}{*}{0.9925} & MILP & 0/3 & 3,600.0 & 83.3 & 88.9 \\
 & & B\&P & 0/3 & 3,600.0 & 100.0 & 0\\
 & & Range-B\&P & 3/3 & 164.5 & 0 & 91.3 \\
 \cmidrule{2-7}
& \multirow{3}{*}{0.995} & MILP & 2/3 & 1,506.6 & 33.3 & 82.2 \\
 & & B\&P & 0/3 & 3,600.0 & 98.0 & 0\\
 & & Range-B\&P & 3/3 & 56.7 & 0 & 83.0 \\
 \cmidrule{2-7}
& \multirow{3}{*}{0.9975} & MILP & 3/3 & 166.7 & 0 & 63.0 \\ 
 & & B\&P & 0/3 & 3,600.0 & 75.3 & 0\\
 & & Range-B\&P & 3/3 & 10.8 & 0 & 63.0 \\
\midrule
 & & MILP & 25/36 & 1,220.4 & 29.2 & 85.2 \\
\multicolumn{2}{c}{\textbf{Average}} & B\&P & 7/36 & 3,094.4 & 75.7 & 11.3 \\
 & & Range-B\&P & 36/36 & 148.6 & 0 & 87.4 \\
\bottomrule
\end{tabular}
\end{table}

Table~\ref{tab:results_gap} presents the results for all instance configurations. For each combination of number of customers, profit lower bound, and solution method, we present the number of instances solved to optimality, the computing time, optimality gap, and the percentage-wise reduction in range of the best solution found compared to the most cost-efficient solution ($\Delta)$. Per row, the results are averaged over the three instances of type A, B, and C, and the final row averages over all instance configurations. In case not all instances are solved to optimality, the computing times provide lower bounds on the true values. 

The results in Table~\ref{tab:results_gap} once again show the drastic improvement in computational performance following from the use of range branching. The classical branch-and-price implementation is able to solve only 7 out of 36 instances, all of which with 5 agents. When augmenting the branching scheme with our range branching rule, all instances can be solved to optimality and the average computing time is reduced by a factor larger than 20. Range-B\&P also outperforms the commercial MILP-solver applied to the compact formulation: 11 more instances are solved to optimality, and the average computing time decreases by a factor of 8. The average reduction in range is two percentage points higher for range-B\&P, suggesting that strictly better integral solutions are found by this method. This result is mainly driven by the instances with 10 or 20 agents. On instances with 5 agents, the MILP-solver is on average faster than range-B\&P, although computing times are of the same order of magnitude. Across all methods, we observe that a tighter profit lower bound generally leads to a reduction in computing times, logically following from the fact that the feasible region becomes smaller. Interestingly, we observe that the computing time of range-B\&P often decreases as the number of agents grows. This phenomenon can largely be attributed to the lower maximum resource capacity per agent in these instances. As a result, the number of potential assignments is significantly reduced, which in turn decreases the time required to solve each node in the branch-and-price tree. Finally, we note that a large reduction in range can be attained at the cost of only a 1\% profit reduction.  

\section{Order-Based Minimization}
\label{sec:order}

Up to now, we have focused on range minimization problems, where the goal is to minimize the difference between the largest and smallest payoff among a set of columns. From a slightly more general perspective, we can see this as minimizing a linear combination of the ordered payoffs, in which we only place non-zero weight on the first and last ordered value (i.e., 1 and -1). Such order-based objective functions are frequently encountered in fairness-oriented optimization \cite{tsang2023unified}. For example, the Gini deviation is an important fairness measure that penalizes deviations between all pairs of payoffs, not just between the largest and smallest. It generalizes the range and can be written as an order-based objective function. In this section, we show how the range branching rule can be generalized to order-based minimization problems.

The remainder of this section is structured as follows. We provide a formal definition of the order-based minimization problem as well as a mathematical formulation in Sections~\ref{subsec:order_problem} and \ref{subsec:order_formulation}, respectively. We propose order branching, a generalization of range branching, in Section~\ref{subsec:order_branching}, and evaluate its effectiveness by minimizing the Gini deviation in a CVRP application in Section~\ref{subsec:order_results}.

\subsection{Problem Description }
\label{subsec:order_problem}

We formally define the order-based minimization problem by making some additional assumptions on top of the setting described in Section~\ref{subsec:problem_description}. As before, we consider the problem of selecting a feasible subset of columns satisfying some linear constraints, i.e., the feasible region is contained in $\mathcal{X} = \{ \bm{x} \in \{0, 1\}^N : A \bm{x} \leq \bm{b} \}$. In addition, to ensure that the order-based objective is well-defined, we assume that a fixed number of exactly $K$ columns is to be selected in any solution, where we require $K \geq 2$. This could be, for example, the number of available drivers (one per vehicle) in a vehicle routing application, or the number of agents in a generalized assignment problem. The feasible region is now given by $\mathcal{X}' = \{ \bm{x} \in \mathcal{X} : \bm{1}^\top \bm{x} = K\} \subseteq \mathcal{X}$. For any $\bm{x} \in \mathcal{X}'$, let $\bm{z}(\bm{x}) \in \mathbb{R}^K$ denote the vector of ordered payoffs, such that $z_k(\bm{x})$ is the value of the $k$-th largest payoff of columns in $\bm{x}$. Let $\bm{v} \in \mathbb{R}^K$ be a weight vector on the order variables. The goal of order-based minimization is to select a set of columns $\bm{x} \in \mathcal{X}'$ that minimizes $\bm{v}^\top \bm{z}(\bm{x})$. The class of objective functions considered here encapsulates a large set of fairness measures \cite{tsang2023unified}. In particular, range minimization corresponds to the special case with $\bm{v} = (1, 0, \ldots, 0, -1)^\top$. 

\subsection{Mathematical Formulation}
\label{subsec:order_formulation}

We propose a simple general formulation of the order-based minimization problem. Recall that $I$ denotes the set of column indices. For each $i \in I$ and $k \in [K]$, introduce the auxiliary binary variable $y_{ik}$ that equals one if column $i$ is selected and yields the $k$-th largest payoff, and zero otherwise. Moreover, let continuous variable $z_k$ denote the value of the $k$-th largest payoff. The order-based minimization problem can then be formulated as follows:
\begin{subequations}\label{eq:order-based}
\begin{align}
\min_{\bm{x}, \bm{y}, \bm{z}} \quad & \bm{v}^\top \bm{z} \\
\text{s.t.}\quad & A \bm{x} \leq \bm{b} \\ 
& \sum_{k=1}^K y_{ik} = x_i && \forall i \in I \label{eq:order_coupling} \\ 
& \sum_{i \in I} y_{ik} = 1 && \forall k \in [K] \label{eq:order_assign} \\ 
& \sum_{i \in I} p_{i} y_{ik} = z_{k} && \forall k \in [K] \label{eq:order_values} \\
& z_k \geq z_{k+1} && \forall k \in [K - 1] \label{eq:order_ranking} \\ 
& \bm{x} \in \{0, 1\}^{N} \\
& \bm{y} \in \{0, 1\}^{N K}. 
\end{align}
\end{subequations}
Constraints (\ref{eq:order_coupling}) and (\ref{eq:order_assign}) ensure that each selected column takes exactly one position in the list of ordered payoffs, and that each position in the order corresponds to exactly one column. As a result, precisely $K$ columns are selected. Constraints (\ref{eq:order_assign})-(\ref{eq:order_values}) ensure that the values of $z_k$ are correctly set, while constraints (\ref{eq:order_ranking}) guarantee a consistent ordering.

We make some brief remarks on formulation (\ref{eq:order-based}). First, we are not aware of any other formulation for the order-based minimization problem that does not require additional binary variables or big-$M$ constraints. Second, while $\mathcal{O}(NK)$ auxiliary variables are required, this formulation is still relatively compact in the sense that only $\mathcal{O}(N + K)$ extra linear constraints are introduced (aside from bound and integrality constraints). A smaller formulation can be obtained by projecting out $\bm{x}$ and dropping constraints (\ref{eq:order_coupling}). In fact, solving model (\ref{eq:order-based}) with constraints (\ref{eq:order_coupling}) would necessitate simultaneous column-and-row generation. Third, it is possible to strengthen this formulation by adding valid inequalities. In particular, assuming ordered payoffs $p_1 \geq p_2 \geq \ldots \geq p_N$, one can add  $\mathcal{O}(NK)$ constraints of the type
\begin{align}
    \sum_{i=1}^{l} y_{ik} \geq \sum_{i=1}^{l} y_{i, k+1}
\end{align}
for all $k \in [K - 1]$ and $l \in [N]$. In a branch-and-price algorithm, however, the use of such inequalities would imply the generation of additional rows when pricing out columns, severely complicating the pricing problem.\footnote{In the vehicle routing literature, valid inequalities that make the pricing problem significantly harder are referred to as `non-robust cuts'.} Moreover, our focus is on branching strategies rather than the strength of particular formulations. As such, we consider these inequalities to be outside the scope of this paper. 

It is clear that formulation (\ref{eq:order-based}) can be used for range minimization by setting $\bm{v} = (1, 0, \ldots, 0, -1)^\top$. In fact, the vehicle-index formulation (\ref{eq:vehicle_index}) and agent-based formulation (\ref{eq:fgap}) can be seen as special cases of (\ref{eq:order-based}) where $\bm{x}$ is projected out.

\subsection{Order Branching}
\label{subsec:order_branching}

We present an extension of range branching to the case of order-based minimization, which we call \emph{order branching}. As before, we first formalize the notion of \emph{order-violating} solutions, after which we show how to eliminate these through branching.

For notational convenience, given $\bm{y}\in[0,1]^{NK}$ and $\bm{p} \in [0, M]^N$, for each $k \in [K]$ define $p^{k}_{-}(\bm{y}) = \min \limits_{j = 1, \ldots, k} \min \limits_{i : y_{ij} > 0} p_i$ and $p^{k}_{+}(\bm{y}) = \max \limits_{j = k, \ldots, K} \max\limits_{i : y_{ij} > 0} p_i$ to be the minimum and maximum payoff of any column assigned to order at most or at least $k$, respectively. In this definition, we look at the payoffs of columns beyond index $k$ since a solution can only be order-respecting when the payoffs are in non-increasing order. By definition, we have $p^{k}_{+}(\bm{y}) \geq p^{k}_{-}(\bm{y})$. In fact, equality must hold for any feasible solution $(\bm{x}, \bm{y}, \bm{z})$ to \eqref{eq:order-based}, i.e., it must hold that $z_k = p^{k}_{-}(\bm{y}) = p^{k}_{+}(\bm{y})$ for all $k$. As for range minimization, we can formalize this idea by introducing the notion of \emph{order-respecting solutions}.
\begin{definition}[Order-Respecting Solution] 
Let a vector $(\bar{\bm{x}}, \bar{\bm{y}}, \bm{\bar{z}}) \in [0,1]^{N} \times [0,1]^{NK} \times \bR^{K}$ be given. We say this vector is \emph{order-respecting} (with respect to $\bm{p}$) if, for all $k \in [K]$, it holds that $\bar{z}_k = p^{k}_{+}(\bar{\bm{y}}) = p^{k}_{-}(\bar{\bm{y}})$, and \emph{order-violating} (with respect to $\bm{p}$) otherwise.
\label{def:order_respecting}
\end{definition}
In range-minimization, by optimality it holds that range-violation occurs due to $\eta$ being under- or $\gamma$ being over-estimated. In contrast, order-violation can occur due to $z_k$ being either under- or over-estimated, depending on the sign of entries of $\bm{v}$. In other words, a solution can be order-violating due to $\bar{z}_k>p^{k}_{-}(\bar{\bm{y}})$ or $\bar{z}_k<p^{k}_{+}(\bar{\bm{y}})$ for some $k$.

Similar to range branching, order branching eliminates order-violating solutions by branching on the values of $z_k$. Consider a fractional solution $(\bar{\bm{x}}, \bar{\bm{y}}, \bar{\bm{z}})$ encountered at any node of the branch-and-price tree. When this solution is order-respecting, our branching rule does not create any branches, but proceeds to invoke the problem-specific branching scheme. Suppose now that the solution is order-violating, i.e., there exists some index $k \in [K]$ for which, for example, $z_k > p^{k}_{-}(\bm{y})$. We choose a cutoff value $Z_k \in (p^{k}_{-}(\bm{y}), z_k)$ and create two child nodes. In the left child node, we impose $z_k \leq Z_k$, while, in the right child node, we require $z_k \geq Z_k$. Unlike in range branching, we can now enforce strong locally valid inequalities that forbid columns with payoffs above or below $Z_k$ in \emph{both} child nodes. In particular, in the left child node, we can enforce $\sum_{j \geq k} \sum_{i : p_i > Z_k} y_{ij} = 0$, while, in the right child node, we can impose $\sum_{j \leq k} \sum_{i \in I: p_i < Z_k} y_{ij} = 0$. As before, the valid inequalities are implemented by fixing existing columns to zero and preventing the generation of new columns in the pricing problem. The case $z_{k} < p^{k}_{+}(\bm{y})$ is treated similarly. 

\subsection{Computational Experiments: Gini Deviation CVRP}
\label{subsec:order_results}

An important example of an order-based objective function is the Gini deviation, a frequently used inequity measure \cite{tsang2023unified}. The Gini deviation is defined as $\sum_{i} \sum_{j > i} |z_i - z_j|$, and penalizes the difference between not only the largest and smallest payoffs, but also between intermediate payoffs. It can be represented as an order-based objective function with weight vector entries given by $v_k = (K - 2k + 1)$ for $k=1,\ldots,K$. We evaluate the effectiveness of order branching by minimizing the Gini deviation among the distances of routes in the capacitated vehicle routing problem. 

The set-up of our experiments is highly similar to that of Section~\ref{subsec:cvrp_setup}. We consider the same instances and parameter settings as before, but focus only on the vehicle-indexed formulation. This formulation naturally lends itself to the order-based objective setting, as it is essentially formulation (\ref{eq:order-based}) with $\bm{x}$ projected out. Moreover, it is not immediately clear how to correctly model the Gini deviation in the last-customer formulation. Since all vehicles are homogeneous, without loss of generality, we can assume that vehicle $k$ performs the route with $k$-th largest distance. The pricing problem and algorithm described in Section~\ref{subsec:cvrp_branch} remain valid, although the symmetry-breaking technique presented there can no longer be used. For the sake of completeness, we present the full resulting formulation here:
\begin{subequations}
\begin{align}
\min \quad & \bm{v}^\top \bm{z} \\
\text{s.t.} \quad & \sum_{k \in [K]} \sum_{r \in R} a_{ir} y_{rk} = 1 && \forall i \in C \\
& \sum_{k \in [K]} \sum_{r \in R} c_r y_{rk} \leq B \\
& \sum_{r \in R} y_{rk} = 1 && \forall k \in [K] \\
& \sum_{r \in R} p_r y_{rk} = z_k && \forall k \in [K] \\
& z_k \geq z_{k+1} && \forall k \in [K - 1] \\
& y_{rk} \in \{0, 1\} && \forall k \in [K],~r \in R.
\end{align}
\label{eq:vehicle_order}
\end{subequations}

The results in Table~\ref{tab:results_gini_cvrp} show that the order branching scheme significantly outperforms the classical branching scheme. For all instance sizes, the number of instances solved to optimality grows, and the computing time and optimality gap are reduced, when adding the order branching rule to the classical branching scheme. As for range branching, we observe that this effect is driven by a large reduction in the size of the branch-and-price tree. Moreover, we find that the Gini deviation is reduced by up to 50\% compared to the cost-efficient solution. Interestingly, comparing the results of the vehicle-index formulation for range minimization and order-based minimization, we observe that it is more effective for the latter problem, being able to solve more instances to optimality with the improved branching scheme.

\begin{table}[tb!]
\caption{Computational performance of vehicle-index formulation for Gini deviation with classical and order branching scheme.}
\label{tab:results_gini_cvrp}
\centering 
\renewcommand{\arraystretch}{1}
\begin{tabular}{@{\extracolsep{5pt}}clrrrrr}
\toprule
$|C|$ & Branching & \# Solved & Time (s) & Gap (\%) & \# Nodes & $\Delta$ (\%) \\ 
\midrule
\multirow{2}{*}{15} & Classical & 17/20 & 1,173.7 & 8.5 & 35,079.7 & 39.6 \\ 
& Order & 20/20 & 2.9 & 0 & 86.2 & 46.5 \\ 
\midrule
\multirow{2}{*}{20} & Classical & 3/20 & 3,452.3 & 76.4 & 43,244.9 & 7.5 \\ 
& Order & 20/20 & 55.7 & 0 & 221.9 & 50.7 \\ 
\midrule
\multirow{2}{*}{25} & Classical & 0/20 & 3,600.0 & 99.6 & 19,211.8 & 0.0 \\ 
& Order & 12/20 & 2,078.3 & 18.3 & 340.1 & 38.7 \\  
\bottomrule
\end{tabular}
\end{table}

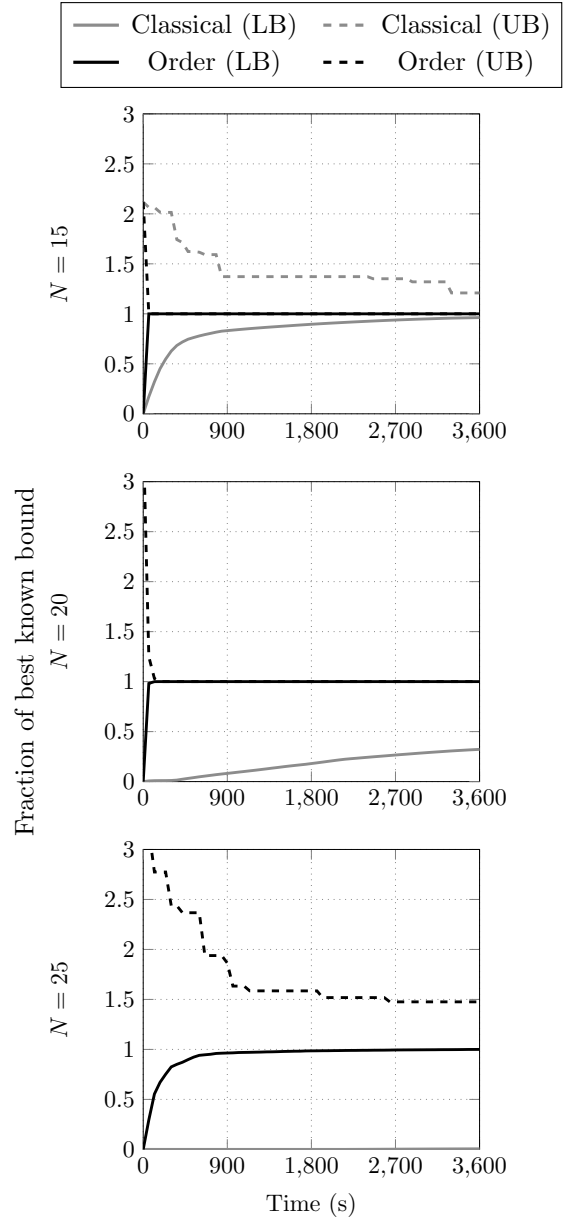
\begin{figure}[tb!]
\centering
\begin{tikzpicture}[scale=0.9]
    \begin{groupplot}[group style={
                      group name=myplot,
                      group size= 1 by 3},
                      height=6cm,
                      width=6.5cm]
    \pgfplotsset{group/every plot/.append style={
        ymin=0, ymax=3, xmin=0, xmax=3600, extra x ticks={0}, extra y ticks={0}, grid, grid style={opacity = 0.5, dotted, black}, disabledatascaling, xtick={900, 1800, 2700, 3600}, ytick={0.5, 1.0, 1.5, 2.0, 2.5, 3.0}}}
        
    \nextgroupplot[ylabel={$N=15$}]
    \addplot[very thick, color=black!45] table[x index=0, y index=1, col sep=semicolon] {ResultsGini/bounds_15_0.csv};\label{plots:plot5}

    \addplot[very thick, dashed, color=black!45]
     table[x index=0, y index=2, col sep=semicolon] {ResultsGini/bounds_15_0.csv}; \label{plots:plot6}

    \addplot[very thick, color=black] table[x index=0, y index=1, col sep=semicolon] {ResultsGini/bounds_15_1.csv};\label{plots:plot7}

    \addplot[very thick, dashed, color=black] table[x index=0, y index=2, col sep=semicolon] {ResultsGini/bounds_15_1.csv};\label{plots:plot8}

    \nextgroupplot[ylabel={$N=20$}]
    \addplot[very thick, color=black!45] table[x index=0, y index=1, col sep=semicolon] {ResultsGini/bounds_20_2.csv};

    \addplot[very thick, dashed, color=black!45]
     table[x index=0, y index=2, col sep=semicolon] {ResultsGini/bounds_20_2.csv};

    \addplot[very thick, color=black] table[x index=0, y index=1, col sep=semicolon] {ResultsGini/bounds_20_3.csv};

    \addplot[very thick, dashed, color=black] table[x index=0, y index=2, col sep=semicolon] {ResultsGini/bounds_20_3.csv};
       
    \nextgroupplot[xlabel={Time (s)}, ylabel={$N=25$}]
    \addplot[ultra thick, color=black!45] table[x index=0, y index=1, col sep=semicolon] {ResultsGini/bounds_25_4.csv};

    \addplot[very thick, dashed, color=black!45]
     table[x index=0, y index=2, col sep=semicolon] {ResultsGini/bounds_25_4.csv};
     
    \addplot[very thick, color=black] table[x index=0, y index=1, col sep=semicolon] {ResultsGini/bounds_25_5.csv};

    \addplot[very thick, dashed, color=black] table[x index=0, y index=2, col sep=semicolon] {ResultsGini/bounds_25_5.csv};
    \end{groupplot}
    \path (myplot c1r1.outer north west)
          -- node[anchor=south,rotate=90] {Fraction of best known bound}
          (myplot c1r3.outer south west);
\path (myplot c1r1.north west|-current bounding box.north)--
      coordinate(legendpos)
      (myplot c1r1.north east|-current bounding box.north);
\matrix[
    matrix of nodes,
    anchor=south,
    draw,
    inner sep=0.2em,
    draw
  ]at([yshift=1ex]legendpos)
  {\ref{plots:plot5}& Classical (LB)&[5pt]
\ref{plots:plot6}& Classical (UB)\\
\ref{plots:plot7}& Order (LB)&[5pt]
\ref{plots:plot8}& Order (UB)\\};
\end{tikzpicture}
\caption{Progression of lower and upper bound of vehicle-index formulation for Gini deviation with the classical and order branching scheme.}
\label{fig:results_gini_bounds}
\end{figure}

Figure~\ref{fig:results_gini_bounds} presents the progression of upper and lower bounds throughout the course of the branch-and-price algorithm. As with range branching, we find that order branching is extremely effective at quickly driving up the lower bound as compared to a classical branching scheme. In addition, we are able to identify high-quality integer solutions earlier on, leading to an improved upper bound development. 

We have also experimented with using formulation (\ref{eq:vehicle_order}) to solve the range minimization problem. Interestingly, the order-based formulation seems to slightly outperform the vehicle-index formulation (\ref{eq:vehicle_index}), but is not as good as the last-customer formulation \eqref{eq:customer_formulation}.

\section{Conclusion}
\label{sec:conclusion}

In this work, we proposed a generic branch-and-price approach for range minimization problems. Our approach relies on range branching, a branching rule that effectively cuts off fractional nodes underestimating the true range of the solution. We have shown several desirable properties of the branching rule, and confirmed its strong computational performance in extensive experiments on instances of the fair capacitated vehicle routing problem and fair generalized assignment problem. Using our new branching rule, we were able to solve many more large instances to optimality, compared to using classical branching strategies. Finally, we successfully generalized our approach to the case of order-based objective functions, such as the Gini deviation.

Depending on the problem at hand, some critical remarks on range branching are in place. First, minimizing the range can have unintended negative consequences in the form of `wasteful' columns being selected in optimal solutions. For example, in our capacitated vehicle routing application we observed that the range was reduced through the use of TSP-suboptimal routes. The degree to which this is problematic and can be avoided is problem-specific. Second, imposing the range branching decisions can make the pricing problem degenerate. This was observed in the capacitated vehicle routing problem, where imposing our branching strategy caused a weakening of label dominance rules. The issue did not occur in the generalized assignment problem, where the efficiency of the pricing algorithm was not affected by range branching decisions. Moreover, the increased computing time per branch-and-price node might still be compensated by a large reduction in the number of nodes. Once again, the net effect of applying range branching is largely problem-specific.

Several promising directions for future research remain. First, it would be interesting to apply our approach to other problems, such as railway crew scheduling or airline crew pairing. Second, Section~\ref{subsec:branching_practical} contains many suggestions for different ways of implementing range branching, including, for example, layered branching schemes with increasingly tight cutoff values. Third, our approach can be readily applied to problems where the objective consists of minimizing the maximum payoff, or maximizing the minimum payoff. Fourth, efficiently ensuring TSP-optimality of selected routes in the CVRP remains an interesting open problem. Finally, while range branching focuses on improving the lower bound throughout the branch-and-price algorithm, it seems promising to combine this with heuristics aimed at providing tight upper bounds.

\section{Acknowledgements}

We are grateful for the fruitful suggestions of three anonymous reviewers. We also thank dr. Twan Dollevoet for his valuable feedback on earlier versions of this manuscript. 

\bibliography{sn-bibliography}

\end{document}